\documentclass[reqno]{amsart}

\usepackage{amsmath,amsthm,amsfonts, amssymb,amscd}
\usepackage[all]{xy}
\numberwithin{equation}{section}

\newtheorem{theorem}{Theorem}[section]
\newtheorem{lemma}[theorem]{Lemma}
\newtheorem{definition}[theorem]{Definition}
\newtheorem{example}[theorem]{Example}
\newtheorem{remark}[theorem]{Remark}
\newtheorem{corollary}[theorem]{Corollary}
\newtheorem{proposition}[theorem]{Proposition}

\newcommand{\lex}{\,\overrightarrow{\times}\,}

\newcommand{\Ker}{\mbox{\rm Ker}}

\newcommand{\Con}{\mbox{\rm Con}}
\newcommand{\blex}{\,\overleftarrow{\times}\,}

\begin{document}
\title[Riesz Space-Valued States]{Riesz Space-Valued States on Pseudo MV-algebras}
\author[Anatolij Dvure\v{c}enskij]{Anatolij Dvure\v{c}enskij$^{1,2}$}
\date{}%
\maketitle
\begin{center}  \footnote{Keywords: MV-algebra, pseudo MV-algebra, state, state-morphism, unital Riesz space, $(R,1_R)$-state, extremal $(R,1_R)$-state, $(R,1_R)$-state-morphism, $R$-Jordan signed measure, Choquet simplex, Bauer simplex

 AMS classification: 06D35, 06C15

The paper has been supported by grant of the Slovak Research and Development Agency under contract APVV-16-0073, by the grant VEGA No. 2/0069/16 SAV
and GA\v{C}R 15-15286S. }
Mathematical Institute,  Slovak Academy of Sciences\\
\v Stef\'anikova 49, SK-814 73 Bratislava, Slovakia\\
$^2$ Depart. Algebra  Geom.,  Palack\'{y} Univer.\\
17. listopadu 12, CZ-771 46 Olomouc, Czech Republic\\

E-mail: {\tt
dvurecen@mat.savba.sk}
\end{center}

\begin{abstract}
We introduce Riesz space-valued states, called $(R,1_R)$-states, on a pseudo MV-algebra, where $R$ is a Riesz space with a fixed strong unit $1_R$.  Pseudo MV-algebras are a non-commutative generalization of MV-algebras. Such a Riesz space-valued state is a generalization of usual states on MV-algebras. Any $(R,1_R)$-state is an additive mapping preserving a partial addition in pseudo MV-algebras. Besides we introduce $(R,1_R)$-state-morphisms and extremal $(R,1_R)$-states, and we study relations between them. We study metrical completion of unital $\ell$-groups with respect to an $(R,1_R)$-state. If the unital Riesz space is Dedekind complete, we study when the space of $(R,1_R)$-states is a Choquet simplex or even a Bauer simplex.
\end{abstract}

\section{Introduction}

The notion of a state is a basic one in the theory of quantum structures, see e.g. \cite{DvPu}. It is an analogue of a finitely additive probability measure. MV-algebras as well as its non-commutative generalization, pseudo MV-algebras, introduced in \cite{GeIo,Rac}, form an important subclass of quantum structures. Mundici defined a notion of a state on an MV-algebra in \cite{Mun2} as averaging the truth-value in \L ukasiewicz logic. States on MV-algebras are studied very intensively last 10--15 years when many important results as an integral representation of states by regular $\sigma$-additive probability measures, \cite{Kro,Pan}, or the MV-algebraic approach to de Finetti's notion of coherence, see \cite{KuMu}, have been established. Some applications of states on MV-algebras can be found in \cite{RiMu}.

In the last period, the so-called Riesz MV-algebras have been studied in the frames of MV-algebras, see \cite{DiLe}. A prototypical example of Riesz MV-algebras is an interval in a unital Riesz space, when we use Mundici's representation functor $\Gamma$, see \cite{Mun1,CDM}. The converse is also true: For any Riesz MV-algebra $M$, there is a unital Riesz space $(R,1_R)$  such that $M \cong \Gamma(R,1_R)$, \cite[Thm 3]{DiLe}. Whereas MV-algebras are algebraic semantic of the  \L ukasiewicz logic, \cite{Cha}, Riesz MV-algebras are an extension of the \L ukasiewicz logic: The propositional calculus that has Riesz MV-algebras as models is a conservative
extension of \L ukasiewicz infinite-valued propositional calculus, \cite{DiLe}. Moreover, these structures have also several applications, among which we mention artificial neural networks, image compression, game
theory, etc., see \cite{ADG, KrMa}. Fuzzy logics with noncommutative
conjunctions inspired by pseudo MV-algebras were studied in \cite{Haj}.

For more information about MV-algebras, see \cite{CDM} and about states on MV-algebras, see \cite{Mun3}, and for the most fresh survey on states on MV-algebras, see \cite{FlKr}.

States on pseudo MV-algebras have been studied in \cite{DvuS}. For pseudo MV-algebras there is a basic representation by unital $\ell$-groups not necessarily Abelian, \cite{Dvu1}, which generalizes Mundici's representation of MV-algebras, see \cite{Mun2}.  A state on a pseudo MV-algebra is defined as an additive functional with non-negative real values which at the top element of the MV-algebra attains the value 1. Whereas every MV-algebra (with $0\ne 1$) admits at least one value, this is not a case for pseudo MV-algebras, because as it was shown in \cite{DvuS}, there are stateless pseudo MV-algebras. Moreover, a pseudo MV-algebra admits at least one state if and only if the pseudo MV-algebra has at least one normal ideal that is also maximal. Therefore, every linearly ordered pseudo MV-algebras, representable pseudo MV-algebras or normal valued ones have at least one state.

Riesz space-valued states on MV-algebras have been firstly introduced in \cite{BDV} under the name generalized states. In the present paper, we introduce an $(R,1_R)$-state on a pseudo MV-algebra as a mapping which attains values in the interval $[0,1_R]$ of the unital Riesz space $(R,1_R)$, where $R$ is a Riesz space and $1_R$ is a fixed strong unit of $R$, Section 3. Besides we introduce extremal $(R,1_R)$-states and $(R,1_R)$-state-morphisms as homomorphisms of pseudo MV-algebras into the interval $[0,1_R]$. We show relations between them and we discuss when the latter two kinds of $(R,1_R)$-states coincide and when not. Whereas according to \cite{DvuS}, there is a one-to-one correspondence among extremal states, state-morphisms and maximal ideals that are normal, respectively,  we show that for $(R,1_R)$-states this is not a case, in general. We will study cases when $(R,1_R)$ is an Archimedean unital Riesz space or even a Dedekind complete unital Riesz space. In Section 4, we present metrical completion of a unital $\ell$-group by an $(R,1_R)$-state. In Section 5, we introduce also $R$-measures and $R$-Jordan signed measures and we study situations when the $(R,1_R)$-state space is a simplex, or a Choquet simplex or even a Bauer simplex and when every $(R,1_R)$-state lies in the weak closure of the convex hull of extremal $(R,1_R)$-states.

The paper is endowed with a couple of illustrating examples.

\section{Pseudo MV-algebras and Riesz Spaces}

In the present section we gather basic notions and results on pseudo MV-algebras and Riesz spaces.

Pseudo MV-algebras as a non-commutative generalization of MV-algebras were defined independently in \cite{GeIo} as pseudo MV-algebras and in \cite{Rac} as generalized MV-algebras.

\begin{definition}\label{de:2.1}     
A {\it pseudo MV-algebra} is an algebra $(M;\oplus,^-,^\sim,0,1)$ of type $(2,1,1,$ $0,0)$ such that the following axioms hold for all $x,y,z \in M$ with an additional binary operation $\odot$ defined via $$ y \odot x =(x^- \oplus y^-) ^\sim $$
\begin{enumerate}
\item[{\rm (A1)}]  $x \oplus (y \oplus z) = (x \oplus y) \oplus z$;

\item[{\rm (A2)}] $x\oplus 0 = 0 \oplus x = x$;

\item[{\rm (A3)}] $x \oplus 1 = 1 \oplus x = 1$;

\item[{\rm (A4)}] $1^\sim = 0;$ $1^- = 0$;

\item[{\rm (A5)}] $(x^- \oplus y^-)^\sim = (x^\sim \oplus y^\sim)^-$;

\item[{\rm (A6)}] $x \oplus x^\sim \odot y = y \oplus y^\sim
\odot x = x \odot y^- \oplus y = y \odot x^- \oplus x$;

\item[{\rm (A7)}] $x \odot (x^- \oplus y) = (x \oplus y^\sim)
\odot y$;

\item[{\rm (A8)}] $(x^-)^\sim= x$.
\end{enumerate}
\end{definition}

We shall assume that $0 \ne 1$.
If we define $x \le y$ iff $x^- \oplus y=1$, then $\le$ is a
partial order such that $M$ is a distributive lattice with $x
\vee y = x \oplus (x^\sim \odot y)$ and $x \wedge y =  x \odot
(x^- \oplus y)$. We recall that a pseudo MV-algebra is an MV-algebra iff $\oplus$ is a commutative binary operation. As usually, we assume that $\odot$ has higher binding priority than $\wedge$ and $\oplus$, and $\oplus$ is higher than $\vee$.

A non-void subset $I$ of $M$ is an {\it ideal} of $M$ if (i) $a\le b\in I$ implies $a \in I$, and (ii) if $a,b \in I$, then $a\oplus b \in I$. The sets $M$ and $\{0\}$ are ideals of $M$. An ideal $I\ne M$ of $M$ is {\it maximal} if it is not a proper subset of any proper ideal of $M$. An ideal $I$ of $M$ is {\it normal} if $a\oplus I:=\{a\oplus b\colon b \in I\}= \{c \oplus a\colon c\in I\}=: I\oplus a$ for any $a\in M$.
For basic properties of pseudo MV-algebras see \cite{GeIo}.

Pseudo MV-algebras are intimately connected with $\ell$-groups. We remind that a {\it po-group} is a group $(G;+,-,0)$ written additively endowed with a partial order $\le$ such that, for $g,h \in G$ with $g\le h$ we have $a+g+b\le a+h+b$ for all $a,b \in G$. If the partial order $\le$ is a lattice order, $G$ is said to be an $\ell$-{\it group}. The {\it positive cone} of a po-group $G$ is the set $G^+=\{g\in G\colon 0\le g\}$. A po-group $G$ satisfies {\it interpolation} if, for $x_1,x_2,y_1,y_2\in G$ with $x_1,x_2\le y_1,y_2$, there is an element $z\in G$ such that $x_1,x_2\le z\le y_1,y_2$.
An element $u\ge 0$ is a {\it strong unit} of $G$ if, given $g\in G$, there is an integer $n\ge 1$ such that $g\le nu$. A couple $(G,u)$, where $G$ is an $\ell$-group and $u$ is a fixed strong unit of $G$, is said to be a {\it unital} $\ell$-{\it group}. An $\ell$-group $G$ is (i) {\it Archimedean} if, for $a,b\in G$, $na\le b$ for each integer $n\ge 1$ implies $a\le 0$, (ii) {\it Dedekind} $\sigma$-{\it complete} if any  sequence $\{g_n\}$ of elements of $G$ that is bounded from above by some element $g_0\in G$ has supremum $\bigvee_n g_n \in G$, and (iii) {\it Dedekind complete} if any family $\{g_t: t \in T\}$ of elements of $G$ which is bounded from above by some element $g_0\in G$ has supremum $\bigvee_{t\in T} g_t \in G$. An $\ell$-{\it ideal} of an $\ell$-group $G$ is any $\ell$-subgroup $P$ of $G$ such that $a\in P$ and $|a|\le |b|$ yield $a \in P$. Here  $|g|= g^+ + g^-$, $g^+= g\vee 0$ and $g^-=-(g\wedge 0)$ for each $g \in G$. For non-explained notions about $\ell$-groups, please, consult e.g. \cite{Fuc, Gla}.

A prototypical example of pseudo MV-algebras is from $\ell$-groups:
If $u$ is a strong unit of a (not necessarily Abelian)
$\ell$-group $G$,
$$\Gamma(G,u) := [0,u]
$$
and
\begin{align*}
x \oplus y &:=
(x+y) \wedge u, \\
x^- &:= u - x,\\
x^\sim &:= -x +u,\\
x\odot y&:= (x-u+y)\vee 0,
\end{align*}
then $\Gamma(G,u):= ([0,u];\oplus, ^-,^\sim,0,u)$ is a pseudo MV-algebra \cite{GeIo}. Conversely, for every pseudo MV-algebra $M$, there is a unique unital $\ell$-group $(G,u)$ (up to isomorphism of unital $\ell$-groups) such that $M \cong \Gamma(G,u)$, and there is a categorical equivalence between the category of pseudo MV-algebras and the category of unital $\ell$-groups as it follows from the basic representation theorem \cite{Dvu1} given by the functor $(G,u)\mapsto \Gamma(G,u)$.

We define a partial operation, $+$, on $M$ in such a way: $x+y$ is defined in $M$ iff $x\odot y=0$, and in such a case, we set $x+y:=x\oplus y$. Using representation of pseudo MV-algebras by unital $\ell$-groups, we see that the partial operation $+$ coincides with the group addition restricted to $M$. The operation $+$ is associative. We note, that if $x\le y$ for $x,y \in M$, then there are two unique elements $z_1,z_2\in M$ such that $z_1+x=y=x+z_2$. We denote them by $z_1=y-x$ and $z_2=-x+y$, and using the group representation, $-$ coincides with the group subtraction. Since $+$ is associative, we say that a finite system $(a_i)_{i=1}^n$ of $M$ is {\it summable} if there is an element $a=a_1+\cdots+a_n\in M$; $a$ is said to be the {\it sum} of $(a_i)_{i=1}^n$ and the sequence $(a_i)_{i=1}^n$ is said to be {\it summable}.

Since there is a categorical equivalence between pseudo MV-algebras and unital $\ell$-groups, \cite{Dvu1}, for the partial addition $+$ on any pseudo MV-algebra the following form of the Riesz Decomposition property, called the {\it strong Riesz Decomposition Property}, RDP$_2$ for short, holds: If for any $a_1,a_2, b_1, b_2 \in M$ such that $a_1+a_2 = b_1+b_2$ there are four elements $c_{11}, c_{12}, c_{21}, c_{22} \in M$ such that $a_1=c_{11}+c_{12}$, $a_2=c_{21}+c_{22}$, $b_1=c_{11}+c_{21}$, $b_2=c_{12}+c_{22}$ and $c_{12}\wedge c_{21}=0$.
It is derived from such a decomposition holding in $\ell$-groups, see \cite[Thm V.1]{Fuc} and \cite{DvVe1,DvVe2}.
Equivalently, if $a_1,\ldots,a_m$ and $b_1,\ldots,b_n$ are elements of a pseudo MV-algebra $M$ such that $a_1+\cdots+a_m=b_1+\cdots+b_n$, there is a system $\{c_{ij}\colon 1\le i\le m, 1\le j \le n\}$ of elements of $M$ satisfying
\begin{equation}\label{eq:RDP1}
a_i=c_{i1}+\cdots+c_{in},\quad b_j = c_{1j}+\cdots+c_{mj},
\end{equation}
for all $1\le i\le m$, $1\le j\le n$, and
\begin{equation}\label{eq:RDP2}
(c_{i+1,j}+\cdots +c_{mj})\wedge (c_{i,j+1}+\cdots+ c_{in})=0, \quad i<m, j<n.
\end{equation}

For any $x\in M$ and any integer $n\ge 0$, we define
\begin{align*}
&x^0=1, x^{n+1}=x^n\odot x,\ n\ge 1,\\
&0\odot x=0, (n+1)\odot x= (n\odot x)\odot x,\ n\ge 1,\\
&0x=0, (n+1)x=nx+x, n\ge 1, \text{ if } nx + x \text{ exists in } M.
\end{align*}

A real vector space $R$ with a fixed partial order $\le$ is a {\it Riesz space} if
\begin{itemize}
\item[{\rm (i)}] $R$ with respect to $\le$ is an $\ell$-group;
\item[{\rm (ii)}] $f\in R^+$ implies $af\in R^+$ for every real number $a\ge 0$.
\end{itemize}

A Riesz space $R$ is {\it Archimedean} if it is Archimedean as an $\ell$-group, analogously $R$ is {\it Dedekind} $\sigma$-{\it complete} or {\it Dedekind complete} if so is $R$ as an $\ell$-group. We note that if $R$ is Dedekind complete, then it is Dedekind $\sigma$-complete, and if $R$ is Dedekind $\sigma$-complete then it is Archimedean. A {\it Riesz ideal} of $R$ is any $\ell$-ideal of $R$. We note that any Riesz ideal of $R$ is a Riesz subspace of $R$.

A {\it unital Riesz space} is a couple $(R,1_R)$ where $R$ is a Riesz space and $1_R$ is a fixed strong unit of $R$. Important examples of unital Archimedean Riesz spaces are spaces of real-valued functions on some topological space: Let $T\ne \emptyset$ be a compact Hausdorff topological space. We denote by $C(T)$ and $C_b(T)$ the system of all continuous real-valued functions of $T$ and the system of all bounded real-valued functions on $T$, respectively. Then $C(T)$ and $C_b(T)$ are Archimedean Riesz spaces with respect to the partial order of functions $f\le g$ iff $f(t)\le g(t)$ for each $t \in T$. The function $1_T$ defined by $1_T(t)=1$ for each $t \in T$ is a strong unit for both $C(T)$ and $C_b(T)$. Both spaces are closed under usual product of two functions, so that both spaces are Banach algebras. The space $C(T)$ has an important property: If $T'$ is another non-void compact Hausdorff space, then there is an isomorphism of Riesz spaces between $(C(T),1_T)$ and $(C(T'),1_{T'})$ preserving product of functions iff $T$ and $T'$ are homeomorphic, see \cite[Thm IV.6.26]{DuSc}. In addition, let $T'=\mathcal M(C(T))$ denote the space of maximal ideals of $C(T)$. Under the hull-kernel topology, $T'$ is homeomorphic to $T$ and $C(T)$ and $C(T')$ are isometric and isomorphic Riesz spaces, see \cite[Ex 45.7]{LuZa}.

In the last period, there has appeared a class of important MV-algebras, Riesz MV-algebras, which are connected with Riesz spaces, for more details, see \cite{DiLe}. We note that if $(R,1_R)$ is a unital Riesz space, then the MV-algebra $\Gamma(R,1_R)$ is a characteristic example of a Riesz MV-algebra.

For Archimedean unital Riesz spaces there is a representation theorem by Yosida, see \cite{Yos} or \cite[Thm 45.3]{LuZa}:

\begin{theorem}\label{th:Yos}{\rm [Yosida Theorem]}
Let $(R,1_R)$ be an Archimedean unital Riesz space. Then there is a compact Hausdorff topological space $T$ such that $R$  can be embedded as a Riesz subspace into $C(T)$, the Riesz space of continuous real-valued functions on $T$, such that $1_R$ maps to the constant function $1_T$, where $1_T(t)=1$, $t\in T$. In addition, $T$ can be chosen such that the image of the embedding of $R$ into $C(T)$ is uniformly dense in $C(T)$, i.e. the uniform closure of the image of $R$ is $C(T)$.
\end{theorem}

We remind that there are nice topological characterizations, Nakano's theorems, when the Riesz space $C(T)$ ($T\ne \emptyset$ compact and Hausdorff) is Dedekind $\sigma$-complete and Dedekind complete, respectively: (1) $C(T)$ is Dedekind $\sigma$-complete iff $T$ is a {\it basically disconnected space}, that is, the closure of every open $F_\sigma$ subset of $T$ is open, see e.g. \cite[Cor 9.3]{Goo}, \cite[Thm 43.9]{LuZa}. (2) $C(T)$ is Dedekind complete iff $T$ is {\it extremally disconnected}, that is, the closure of every open set of $T$ is open, see \cite[Thm 43.11]{LuZa}. We note that the same characterizations hold also for the Riesz space $C_b(T)$  of bounded real-valued functions on $T$.

General Dedekind $\sigma$-complete unital Riesz spaces are characterized as follows, see \cite[Thm 45.4]{LuZa}:

\begin{theorem}\label{th:sigma}
If $(R,1_R)$ is a Dedekind $\sigma$-complete Riesz space, then it is isomorphic to some $(C(T),1_T)$, where $T\ne \emptyset$ is a compact basically disconnected Hausdorff topological space.
\end{theorem}

In addition, if $(R,1_R)$ is even Dedekind complete, then $T$ is extremally  disconnected. In both cases, the space $T$ can be chosen as the set of maximal ideals of $R$ topologized by the hull-kernel topology.

More about Riesz spaces can be found in \cite{LuZa} and some information about representations of Archimedean Riesz spaces by systems of functions attaining also infinite values are in the survey \cite{Fil}.

\section{$(R,1_R)$-states on Pseudo MV-algebras}

In the present section, we define states on pseudo MV-algebras and then we define $(R,1_R)$-states as additively defined mappings on a pseudo MV-algebra $M$ which preserve the partial addition $+$ on $M$ and have values in the interval $[0,1_R]$ of a unital Riesz space $(R,1_R)$ mapping the top element $1\in M$ onto the strong unit $1_R\in R$. We introduce also extremal $(R,1_R)$-states, $(R,1_R)$-state-morphisms, and we show relationships between them.

States, analogues of  finitely additive measures, on pseudo MV-algebras were introduced in \cite{DvuS} as follows:
Let $M$ be a pseudo MV-algebra. A {\it state} on $M$ is any real-valued mapping $s: M \to [0,1]$ such that (i) $s(1)=1$, and (ii) $s(x+y)=s(x)+s(y)$ whenever $x+y$ is defined in $M$. According to \cite[Prop 4.1]{DvuS}, if $s$ is a state on $M$, then (i) $s(0)=0$, (ii) $s(a)\le s(b)$ if $a\le b$, (iii) $s(x^-)=1-s(x)= s(x^\sim)$, (iv) $s(x^=)=s(x)=s(x^\approx)$, (v) $s(x\vee y)+s(x\wedge y)= s(x)+s(y)=s(x\oplus y)+s(x\odot y)$, (vi) $s(x\oplus y)=s(y\oplus x)$. A state $s$ is {\it extremal} if from $s=\lambda s_1 + (1-\lambda)s_2$ for states $s_1,s_2$ on $M$ and $\lambda\in (0,1)$ we have have $s_1=s_2$. Let $\mathcal S(M)$ and $\mathcal S_\partial(M)$ denote the set of all states and extremal states, respectively, on $M$. It can happen that $\mathcal S(M)=\emptyset$, see \cite[Cor 7.4]{DvuS}, however, if $M$ is an MV-algebra, $M$ has at least one state, \cite[Cor 4.4]{Goo}. We note that a pseudo MV-algebra possesses at least one state iff $M$ has at least one maximal ideal that is also normal, see \cite{DvuS}. We say that a net $\{s_\alpha\}_\alpha$ of states on $M$ {\it converges weakly} to a states $s$ if $s(x)=\lim_\alpha s_\alpha(x)$ for each $x \in M$. Then $\mathcal S(E)$ and $\mathcal S_\partial(M)$ are either simultaneously the empty sets or non-void compact Hausdorff topological spaces, and due to the Krein--Mil'man Theorem, every state $s$ on $M$ is a weak limit of a net of convex combinations of extremal states.

Now we extend the notion of a state to a Riesz space-valued mapping.

\begin{definition}\label{de:3.1}
Let $1_R$ be a strong unit of a Riesz space $R$. An $(R,1_R)$-{\it state} on a pseudo MV-algebra $M$ is any mapping $s:M\to [0,1_R]$ such that {\rm (i)} $s(1)=1_R$ and {\rm (ii)} $s(x+y)=s(x)+s(y)$ whenever $x+y$ is defined in $M$.
\end{definition}

An $(R,1_R)$-{\it state-morphism} on a pseudo MV-algebra $M$ is any homomorphism of pseudo MV-algebras $s:M \to \Gamma(R,1_R)$. We denote by $\mathbb R$ the group of real numbers.
It is evident that any $(\mathbb R,1)$-state is a state on $M$.
An $(\mathbb R,1)$-state-morphism on $M$ is said to be a {\it state-morphism}.
It is clear that any $(R,1_R)$-state-morphism is in fact an $(R,1_R)$-state on $M$. The converse is not true, in general.

We can define also an $(R,1_R)$-{\it state} on every unital $\ell$-group $(G,u)$ as follows: It is a mapping $s: G\to R$ such that (i) $s(g)\ge 0$ if $g\ge 0$, (ii) $s(g+h)=s(g)+s(h)$ for all $g,h \in G$, and (iii) $s(u)=1_R$. The restriction of any $(R,1_R)$-state on $(G,u)$ onto the pseudo MV-algebra $\Gamma(G,u)$ gives an $(R,1_R)$-state on $\Gamma(G,u)$, and using the categorical equivalence between pseudo MV-algebras and unital $\ell$-groups, see \cite[Thm 6.4]{Dvu1}, every $(R,1_R)$-state on $\Gamma(G,u)$ can be extended to a unique $(R,1_R)$-state on $(G,u)$.

It is worthy of recalling that if $s$ is an $(R,1_R)$-state on $\Gamma(S,1_S)$, where $(S,1_S)$ is a unital Riesz space, then $s(tx)=ts(x)$
for each $x \in \Gamma(S,1_S)$ and any real number $t \in [0,1]$. Indeed, since $s(x)=s(n \frac{1}{n}x)=ns(\frac{1}{n}x)$, i.e. $s(\frac{1}{n}x)= \frac{1}{n}s(x)$. Then for each integer $m=0,1,\ldots,n$, we have $s(\frac{m}{n}x)=s(m\frac{1}{n}x)=ms(\frac{1}{n}x)= \frac{m}{n}s(x)$. The statement is trivially satisfied if $t=0,1$. Thus let $t \in (0,1)$. There are two sequences of rational numbers $\{p_n\}$ and $\{q_n\}$ from the interval $(0,1)$ such that $\{s_n\}\nearrow t$ and  $\{q_n\}\searrow t$ which implies $p_ns(x)= s(p_nx)\le s(tx) \le s(q_nx)=q_ns(x)$, so that $s(tx)=ts(x)$.

In addition, if $s$ is an $(R,1_R)$-state on a unital Riesz space $(S,1_S)$, we can show that $s(ta)=ts(a)$ for each $a \in S$ and $t \in \mathbb R$.

\begin{proposition}\label{pr:3.2}
Let $s$ be an $(R,1_R)$-state on a pseudo MV-algebra $M$. Then
\begin{itemize}
\item[{\rm (i)}] $s(0)=0$.
\item[{\rm (ii)}] If $x\le y$, then $s(x)\le s(y)$, and
$$
s(y\odot x^-)=s(y)-s(x)=s(x^\sim \odot y).
$$
\item[{\rm (iii)}] $s(x^-)=1-s(x)= s(x^\sim)$.
\item[{\rm (iv)}] $s(x^=)=s(x)=s(x^\approx)$.
\item[{\rm (v)}] $s(x\vee y)+s(x\wedge y)= s(x)+s(y)$.
\item[{\rm (vi)}] $s(x\oplus y)+s(x\odot y)=s(x)+s(y)$.
\item[{\rm (vii)}] $s(x\oplus y) \oplus s(x\odot y)=s(x)\oplus s(y)$.
\item[{\rm (viii)}] The kernel of $s$, $\Ker(s):=\{x \in M\colon s(x)=0\}$, is a normal ideal of $M$.
\item[{\rm (ix)}] $[x]=[y]$ if and only if $s(x)=s(x\wedge y)=s(y)$, where where $[x]$ and $[y]$ are the cosets in $M/\Ker(s)$ determined by $x,y \in  M$.
\item[{\rm (x)}] There is a unique $(R,1_R)$-state $\tilde s$ on $M/\Ker(s)$ such that $\tilde s([x]) = s(x)$ for each $[x] \in M/\Ker(s)$.

\item[{\rm (xi)}] $\tilde s([x])=0$ if and only if $[x]=[0]$.

\item[{\rm (xii)}] $s(x\oplus y)=s(y\oplus x)$ whenever $R$ is an Archimedean Riesz space. In addition,  $M/\Ker(s)$ is an Archimedean MV-algebra.
\end{itemize}
\end{proposition}

\begin{proof}
Assume $M = \Gamma(G,u)$ for some unital $\ell$-group $(G,u)$.
Properties (i)--(iv) follow directly from definition of pseudo MV-algebras and $(R,1_R)$-states.

(v) It follows from equalities $(x\vee y)\odot y^- =(x\vee y)-y = x-(x\wedge y)=x\odot (x\wedge y)^-$ which hold in the $\ell$-group $G$ and the pseudo MV-algebra $M$.

(vi) It follows from the identity $x = (x\oplus y)\odot y^- + (y\odot x)$, see \cite[Prop 1.25]{GeIo} and (ii).

(vii) It follows from (vi) and from the identity $r_1\oplus r_1 = (r_1+r_2)\wedge 1_R$ for $r_1,r_2 \in [0,1_R]$.

(viii) If $s(x),s(y)=0$, then by (vi), we have $s(x\oplus y) =0$. By (ii) we conclude $\Ker(s)$ is an ideal of $M$. To show that $\Ker(s)$ is normal, let $a\in M$ and $x \in \Ker(s)$. Then $s(a\oplus x)=s(a)=s(x\oplus a)$ and $a\oplus x = (a\oplus x)\odot a^- \oplus a$ so that, $s((a\oplus x)\odot a^-)=0$. In a similar way, we prove that $x\oplus a= a\oplus (a^\sim \odot (x\oplus a))$ and $a^\sim \odot (x\oplus a))\in \Ker(s)$.

(ix) It is evident.

(x)  Let $[x] \le [y]^-$. We define $x_0 = x\wedge y^-$. Then
$x_0 \le y^-$ and $[x_0] = [x\wedge y^-]$, so that

\begin{align*}
\tilde s([x] + [y]) &= \tilde s([x \oplus y]) = \tilde s([x_0
+ y]) = s(x_0 + y) = s(x_0) + s(y)\\
& = \tilde s([x_0]) + \tilde s([y]) = \tilde s([x]) + \tilde s ([y])
\end{align*}
which proves that $\tilde s$ is an $(R,1_R)$-state on $M/\Ker(s)$. By (ix), $[x]=[y]$ implies $s(x)=s(y)$.

(xi) It follows from (ix).

(xii) Due to (xi), $\tilde s([x]) = 0$ iff $[x] = [0]$.
We claim that $M/\Ker(s)$ is an Archimedean pseudo MV-algebra. Indeed, let $n[x]$ be
defined in $M/\Ker(s)$ for any integer $n \ge 1$. Then $\tilde
s(n[x]) = n \tilde s([x]) = n\, s(x) \le 1_R$ for any $n$.
Therefore, $\tilde s([x]) = s(x) =0.$ The Archimedeanicity of
$M/\Ker(s)$ entails the commutativity of $M/\Ker(s)$, see
\cite[Thm 4.2]{DvuS}. Therefore, $s(x \oplus y) =\tilde s([x\oplus y])=\tilde s([x]\oplus [y]) =\tilde s([y] \oplus [x])=\tilde s([y\oplus x]) = s(y \oplus x)$.
\end{proof}

In the same way as for states, we can define extremal $(R,1_R)$-states. Let $\mathcal S(M,R,1_R)$, $\mathcal{SM}(M,R,1_R)$, and $\mathcal S_\partial(M,R,1_R)$ denote the set of $(R,1_R)$-states, $(R,1_R)$-state-morphisms and extremal $(R,1_R)$-states, respectively, on $M$. Analogously, we can define extremal $(R,1_R)$-states on unital $\ell$-groups. In addition, using the categorical equivalence \cite[Thm 6.4]{Dvu1}, if $M =\Gamma(G,u)$, then an $(R,1_R)$-state $s$ on $M$ is extremal iff the unique extension of $s$ to an $(R,1_R)$-state on $(G,u)$ is extremal and vice-versa, that is, an $(R,1_R)$-state on $(G,u)$ is extremal iff its restriction to $\Gamma(G,u)$ is extremal.

\begin{lemma}\label{le:3.3}
An $(R,1_R)$-state $s$ on $M$ is extremal if and only if $\tilde s$ is extremal on $M/\Ker(s)$.
\end{lemma}

\begin{proof}
(1) Let $s$ be extremal and let $\tilde s = \lambda m_1 + (1-\lambda)m_2$, where $m_1,m_2$ are $(R,1_R)$-states on $M/\Ker(s)$ and $\lambda \in (0,1)$. Then $s_i(x):=m_i([x])$, $x\in M$, $i=1,2$, is an $(R,1_R)$-state on $M$, and $s=\lambda s_1 +(1-\lambda)s_2$ so that $s_1=s_2$ and $m_1=m_2$ proving $\tilde s$ is extremal.

Conversely, let $\tilde s$ be extremal and let $s=\lambda s_1 +(1-\lambda)s_2$ for $(R,1_R)$-states $s_1,s_2$ on $M$ and $\lambda \in (0,1)$. Then $\Ker(s)=\Ker(s_1)\cap \Ker(s_2)$. We assert that $m_i([x]):=s_i(x)$, $[x]\in M/\Ker(s)$, is an $(R,1_R)$-state on $M/\Ker(s)$ for $i=1,2$. Indeed,  first we show that $m_i$ is correctly defined. Thus let $[x]=[y]$. By (ix) of Proposition \ref{pr:3.2}, we have $s(x)=s(x\wedge y)=s(y)$ so that $s(x\odot y^-)=0=s(y\odot x^-)$ which yields $s_i(x\odot y^-)=0=s_i(y\odot x^-)$ and $s_i(x)=s_i(x\wedge y)=s_i(y)$ for $i=1,2$. So that, we have $m_i([x])=m_i([y])$. Therefore, $m_i$ is an $(R,1_R)$-state. Then $\tilde s = \lambda m_1 +(1-\lambda)m_2$ implying $m_1=m_2$ and $s_1=s_2$.
\end{proof}

We note that it can happen that on $M$ there is no $(R,1_R)$-state even for $(R,1_R)=(\mathbb R,1)$ as it was already mentioned. In the following proposition, we show that if $M$ is an MV-algebra, then $\mathcal S(M,R,1_R)$ is non-void.

\begin{proposition}\label{pr:3.4}
Every MV-algebra has at least one $(R,1_R)$-state for every unital Riesz space $(R,1_R)$.
\end{proposition}

\begin{proof}
Due to \cite[Cor 4.4]{Goo}, any MV-algebra $M$ has at least one state; denote it by $s_0$. Then the mapping $s:M \to [0,1_R]$ defined by $s(x):=s_0(x)1_R$, $x \in M$, is an $(R,1_R)$-state on $M$.
\end{proof}

The following result was established in \cite{DvuS} for states on MV-algebras. In the following proposition, we extend it for  $(R,1_R)$-states.

\begin{proposition}\label{pr:3.5}
Let $(R,1_R)$ be a unital Riesz space. The following statements are equivalent:
\begin{itemize}
\item[{\rm (i)}] The pseudo MV-algebra $M$ possesses at least one $(R,1_R)$-state.
\item[{\rm (ii)}] $M$ has has at least one maximal ideal that is normal.
\item[{\rm (iii)}] $M$ has at least one state.
\end{itemize}
{\rm (1)} Every linearly ordered pseudo MV-algebra possesses at least one $(R,1_R)$-state. The same is true if $M$ is representable, i.e. it is representable as a subdirect product of linearly ordered pseudo MV-algebras.

{\rm (2)} If $M$ is a linearly ordered pseudo MV-algebra and $(R,1_R)$ is an Archimedean unital Riesz space, then $M$ possesses only a unique $(R,1_R)$-state.
\end{proposition}

\begin{proof}
By \cite[Prop 4.3]{DvuS}, (ii) and (iii) are equivalent.

(i) $\Rightarrow$ (ii). Let $s$ be an $(R,1_R)$-state on $M$. Since $\Gamma(R,1_R)$ is an MV-algebra, it has at least one state, say $s_0$. Then the mapping $s_0\circ s$ is a state on $M$, which by \cite[Prop 4.3]{DvuS} means that $M$ has at least one maximal ideal that is normal.

(ii) $\Rightarrow$ (i). Let $I$ be a maximal ideal of $M$ that is normal. By \cite[Cor 3.5]{DvuS}, $M/I$ is an MV-algebra. Due to Proposition \ref{pr:3.4}, $M/I$ possesses at least one $(R,1_R)$-state, say $s_0$. Then the mapping $s(x):=s_0(x/I)$, $x \in M$, is an $(R,1_R)$-state on $M$.

(1) Now suppose that $M$ is linearly ordered. By \cite[Prop 5.4]{DvuS}, $M$ possesses a unique maximal ideal and this ideal is normal. Applying just proved equivalences, $M$ possesses an $(R,1_R)$-state.

If $M$ is representable, then it can be embedded into a direct product $\prod_t M_{t\in T}$ of linearly ordered pseudo MV-algebras $\{M_t\colon t\in T\}$, i.e. there is an embedding of pseudo MV-algebras $h:M\to \prod_t M_t$ such that $\pi_t\circ h: M \to M_t$ is a surjective homomorphism, where $\pi_t:\prod_{t \in T}M_t\to M_t$ is the $t$-th projection for each $t \in T$. Every $M_t$ possesses an $(R,1_R)$-state $s_t$, so that $s_t \circ \pi_t\circ h$ is an $(R,1_R)$-state on $M$.

(2) Let $M$ be linearly ordered and $(R,1_R)$ be an Archimedean unital Riesz space. By the Yosida Representation Theorem, Theorem \ref{th:Yos}, there is a compact Hausdorff topological space such that $R$ can be embedded into the Riesz space $C(T)$ of continuous real-valued functions on $T$ as its Riesz subspace. Given $t \in T$, define a mapping $s_t: \Gamma(C(T),1_T)\to [0,1]$ defined by $s_t(f)=f(t)$, $f \in \Gamma(C(T),1_T)$; it is a state on $\Gamma(C(T),1_T)$. Due to (1) of the present proof, $M$ admits at least one $(R,1_R)$-state. Let $s_1,s_2$ be $(R,1_R)$-states on $M$. Then $ s_t\circ s_i$ is a state on $M$ for $i=1,2$ and each $t \in T$. According to \cite[Thm 5.5]{DvuS}, $M$ admits only one state. Therefore, $s_t \circ s_1= s_t \circ s_2$ for each $t \in T$, i.e. $s_1=s_2$.
\end{proof}

For additional relationships between $(R,1_R)$-state-morphisms and their kernels as maximal ideals, see Propositions \ref{pr:3.20}--\ref{pr:3.21} below.

We note that in (2) of the latter proposition, if $(R,1_R)$ is not Archimedean, then it can happen that $M$ has uncountably many $(R,1_R)$-states and each of these states is an $(R,1_R)$-state-morphism, see Example \ref{ex:3.8} below.

\begin{proposition}\label{pr:3.6} {\rm (1)} An $(R,1_R)$-state $s$ on a pseudo MV-algebra $M$ is an $(R,1_R)$-state-morphism if and only if
\begin{equation}\label{eq:3.1}
s(x\wedge y) = s(x)\wedge s(y),\quad x,y \in M.
\end{equation}
{\rm(2)}  An $(R,1_R)$-state $s$ on $M$ is an $(R,1_R)$-state-morphism if and only if the $(R,1_R)$-state $\tilde s$ on $M/\Ker(s)$ induced by $s$ is an $(R,1_R)$-state-morphism on $M/\Ker(s)$.
\end{proposition}

\begin{proof}
(1) Assume that $s$ is an $(R,1_R)$-state-morphism. Then $s(x
\wedge y) = s(x \odot(x^- \oplus y)) = s(x) \odot
(s(x^-) \oplus s(y)) = s(x)\wedge s(y)$.

Conversely, let (\ref{eq:3.1}) hold. Then $x \oplus y = x + (x ^\sim \odot
(x \oplus y)) = x + (x^\sim \wedge y)$, so that $s(x \oplus y) =
s(x) + s(x^\sim \wedge y) = s(x) + (1_R-s(x))\wedge s(y) =
(s(x)+ s(y))\wedge 1_R= s(x)\oplus s(y)$ proving $s$ is an $(R,1_R)$-state-morphism.

(2) Using Proposition \ref{pr:3.2}, we have $s(x\wedge y)=\tilde s([x\wedge y])= \tilde s([x]\wedge [y])$. Applying (1), we have the assertion in question.
\end{proof}

By \cite[Prop 4.3]{DvuS}, a state $s$ on a pseudo MV-algebra $M$ is a state-morphism iff $\Ker(s)$ is maximal. In what follows, we exhibit this criterion for the case of $(R,1_R)$-state-morphisms.

\begin{proposition}\label{pr:3.7}
{\rm (1)} Let $s$ be an $(R,1_R)$-state on $M$. If $\Ker(s)$ is a maximal ideal of $M$, then $s$ is an $(R,1_R)$-state-morphism.

{\rm (2)} In addition, let $(R,1_R)$ be a unital Riesz space such that every element of $R^+\setminus \{0\}$ is a strong unit. If $s$ is an $(R,1_R)$-state-morphism, then $\Ker(s)$ is a maximal ideal.
\end{proposition}

\begin{proof}
(1) Assume $\Ker(s)$ is a  maximal ideal of $M$.
By [GeIo, (i) Prop 1.25], $x \odot y^- \wedge y\odot x^- = 0$ for all $x,y\in M$, that is $[x] \odot [y]^- \wedge [y]\odot [x]^- = [0]$ and due to \cite[Cor 3.5]{DvuS}, $M/\Ker(s)$ is an Archimedean linearly ordered MV-algebra  which entails either
 $s(x\odot y^-) = 0$ or $s(y\odot
x^-)=0$. In the first case we have $0 = s(x \odot y^-) = s(x \odot
(x \wedge y)^-) = s(x) - s(x\wedge y)$. Similarly, $0 = s(y
\odot x^-)$ entails $s(y) - s(x \wedge y) = 0,$ i.e., $s(x\wedge
y) = \min\{s(x), s(y)\}= s(x)\wedge s(y)$, which by Proposition \ref{pr:3.6}
means that $s$ is an $(R,1_R)$-state-morphism.

(2) Let $(R,1_R)$ be a Riesz space such that every strictly positive element of $R$ is a strong unit; then $R$ is Archimedean. Let $s$ be an $(R,1_R)$-state-morphism on $M$ and let $x \in M$ be an element such that $s(x) \ne 0$.

Denote by $\Ker(s)_x$ the ideal of $M$ generated by $\Ker(s)$ and $x$. By \cite[Lem 3.4]{GeIo},
$\Ker(s)_x =\{y \in M:\ y \le n\odot x \oplus h$ for some $n \ge
1$ and some $h \in \Ker(s)\}$. Let $z$ be an arbitrary element of
$M$. Since $s(x)$ is a strong unit of $R$, there exists an integer $n \ge 1$ such that $s(z) < ns(x)$, so that $s(z) \le n\odot s(x)$. Then $s((n\odot x)^\sim \odot z ) = 0$. Since $z = (n\odot x) \wedge z \oplus ((n\odot x)\wedge z)^\sim \odot z= (n\odot x) \wedge z \oplus (n\odot x)^\sim \odot z \le n \odot x \oplus (n\odot x)^\sim \odot z$, it proves that $z \in \Ker(s)_x$, consequently, $M = \Ker(s)_x$ which shows that $\Ker(s)$ is a maximal ideal of $M$.
\end{proof}

We notify that (2) of the preceding proposition follows directly from Theorem \ref{th:3.13}. We have left here the proof of (2) only to present different used methods.

We note that if $s$ is an $(R,1_R)$-state-morphism and $(R,1_R)$ is not Archimedean, then $\Ker(s)$ is not necessarily maximal as the following example shows. In addition, it can happen that every $(R,1_R)$-state is an $(R,1_R)$-state-morphism but not every $(R,1_R)$-state-morphism is extremal.

\begin{example}\label{ex:3.8}
Let $M=\Gamma(\mathbb Z \lex \mathbb Z,(1,0))$, where $\mathbb Z$ is the group of integers, $R=\mathbb R \lex \mathbb R$ be the lexicographic product of the real line $\mathbb R$ with itself, and choose $1_R =(1,0)$. Then $R$ is a linearly ordered Riesz space that is not Archimedean, and every element of the form $(0,x)$, where $x>0$, is  strictly positive but no strong unit for $R$. The mapping $s: M\to [0,1_R]$ defined by $s(a,b)=(a,b)$ for $(a,b)\in M$ is an $(R,1_R)$-state-morphism and $\Ker(s)=\{(0,0)\}$ is an ideal that is not a maximal ideal of $M$ because it is properly contained in the maximal ideal $I=\{(0,n)\colon n \ge 0\}$ which is a unique maximal ideal of $M$.

In addition, $M$ has uncountably many $(R,1_R)$-states, any $(R,1_R)$-state on $M$ is an $(R,1_R)$-state-morphism, and there is a unique $(R,1_R)$-state having maximal kernel and it is a unique extremal $(R,1_R)$-state $M$.
\end{example}

\begin{proof}
Let $s$ be any $(R,1_R)$-state on $M$. Then $s(0,1)=(a,b)$ for some unique $(a,b)\in \Gamma(\mathbb R \lex \mathbb R,(1,0))$, where $a\ge 0$. Since $s(0,n)=(na,nb)\le (1,0)$, $n \ge 0$, we have $a=0$ and $b\ge 0$. Therefore, $s(1,-n)=(1,-nb)$. We denote this $(R,1_R)$-state by $s_b$.  Hence, there is a one-to-one correspondence between $(R,1_R)$-states on $M$ and the positive real axis $[0,\infty)$ given by $b\mapsto s_b$, $b \in [0,\infty)$. Then every $(R,1_R)$-state $s_b$ is an $(R,1_R)$-state-morphism, $\Ker(s_b)=\{(0,0)\}$ for $b >0$ which is an ideal of $M$ but not maximal, and only $s_0$ is an extremal $(R,1_R)$-state on $M$ and $\Ker(s_0)=\{(0,n)\colon n\ge 0\}$ is a maximal ideal.
\end{proof}

We note that in Example \ref{ex:3.18} and Proposition \ref{pr:3.19} we will show also cases of $(R,1_R)$-state-morphisms for an Archimedean Riesz space $(R,1_R)$ whose kernel is not maximal.

We remind the following result from \cite[Lem 4.4]{DvuS} which follows e.g. from \cite[Prop 7.2.5]{CDM}.

\begin{lemma}\label{le:3.10}
{\rm (i)} Let $G_1$ and $G_2$ be two subgroups of $(\mathbb R;+)$ each containing a common non-zero element $g_0$. If there is an injective group-homomorphism $\phi$ of $G_1$ into $G_2$ preserving the order
such that $\phi(g_0) = g_0$, then $G_1 \subseteq G_2$ and
$\phi$ is the identity on $G_1$. If, in addition, $\phi$ is
surjective, then $G_1 = G_2$.

{\rm (ii)} Let $M_1$ and $M_2$ be two MV-subalgebras of the
standard MV-algebra $[0,1]$. If there is an MV-isomorphism
$\psi$ from $M_1$ onto $M_2$, then $M_1 = M_2,$ and $\psi$ is the identity.
\end{lemma}

\begin{theorem}\label{th:3.11}
Let $(R,1_R)$ be an Archimedean unital Riesz space. Let $s_1,s_2$ be two $(R,1_R)$-state-morphisms on a pseudo MV-algebra $M$ such that their kernels are maximal ideals of $M$ and $\Ker(s_1) = \Ker(s_2)$. Then
$s_1 = s_2$.
\end{theorem}

\begin{proof}
Since $R$ is an Archimedean Riesz space with a strong unit $1_R$, due to the Yosida Representation Theorem, Theorem \ref{th:Yos}, there is a compact Hausdorff topological space $T\ne \emptyset$ and an injective homomorphism of Riesz spaces $\phi:R \to C(T)$ with $\phi(1_R)=1_T$, where $(C(T),1_T)$ is the unital Riesz space of continuous real-valued functions on $T$. Then $M_i:=s_i(M)$ are MV-subalgebras of the MV-algebra $\Gamma(R,1_R)$ for $i=1,2$. Define a mapping $s^i_t: \phi(M_i)\to [0,1]$ for each $t \in T$ by $s^i_t(\phi(s_i(x)))= (\phi( s_i(x)))(t)$ for each $t \in T$. Then the mapping $\hat s^i_t:=s^i_t\circ \phi \circ s_i$ is a state-morphism on $M$, and by \cite[Prop 4.3]{DvuS}, each $\Ker(\hat s^i_t)$ is a maximal ideal of $M$. Since $\Ker(s_i)=\bigcap_{t \in T} \Ker(\hat s^i_t)$ and $\Ker(s_i)\subseteq \Ker(\hat s^i_t)$ are also  maximal ideals of $M$, we conclude that $\Ker(\hat s^1_t)=\Ker(s_1)=\Ker(s_2)= \Ker(\hat s^2_t)$ for each $t \in T$ which by \cite[Prop 4.5]{DvuS} means that $s^1_t=s^2_t$ for each $t \in T$. Then $\phi(s_1(x))(t)=\phi(s_2(x))(t)$, $t \in T$, i.e. $s_1 = s_2$.
\end{proof}

We note that if $(R,1_R)$ is not Archimedean, then Theorem \ref{th:3.11} is not necessarily valid, see Example \ref{ex:3.8}.

\begin{proposition}\label{pr:3.12}
Let $I$ be a maximal and normal ideal of a pseudo MV-algebra $M$ and let $(R,1_R)$ be a unital Riesz space. Then there is an $(R,1_R)$-state-morphism $s$ such that $\Ker(s)=I$. If, in addition, $R$ is Archimedean, there is a unique $(R,1_R)$-state-morphism $s$ such that $\Ker(s)=I$.
\end{proposition}

\begin{proof}
Since $M/I$ is by \cite[Prop 3.4]{DvuS} an Archimedean linearly ordered MV-algebra, it is isomorphic by Lemma \ref{le:3.10} to a unique MV-subalgebra of $\Gamma(\mathbb R,1)$; identify it with its image in $\mathbb R$. Define a mapping $s: M \to R$ as follows: $s(x)= x/I1_R$, $x \in M$. Then $s$ is an $(R,1_R)$-state-morphism on $M$ such that $\Ker(s)=I$.

Let, in addition, $R$ be an Archimedean Riesz space. By Theorem \ref{th:3.11}, if $s'$ is another $(R,1_R)$-state-morphism on $M$ with $\Ker(s')=I$, then $s=s'$.
\end{proof}

The following result says about a one-to-one correspondence between $(R,1_R)$-states and states on pseudo MV-algebras for a special kind of Archimedean unital Riesz spaces.

\begin{theorem}\label{th:3.13}
Let $(R,1_R)$ be a unital Riesz space such that every strictly positive element of $R$ is a strong unit for $R$. Given a state $m$ on $M$, the mapping
$$s(x):=m(x)1_R, \quad x \in M,
$$ is an $(R,1_R)$-state on $M$.

Conversely, if $s$ is an $(R,1_R)$-state on a pseudo MV-algebra $M$, there is a unique state $m_s$ on $M$ such
$$
s(x):=m_s(x)1_R,\quad x \in M.
$$
The mapping $s\mapsto m_s$ is a bijective affine mapping from $\mathcal S(M)$ onto $\mathcal S(M,R,1_R)$.

In addition, the following statements are equivalent:
\begin{enumerate}

\item[{\rm (i)}] $s$ is an extremal $(R,1_R)$-state on $M$.

\item[{\rm  (ii)}] $s$ is an $(R,1_R)$-state-morphism on $M$.

\item[{\rm  (iii)}] $s(x\wedge y) =  s(x)\wedge s(y),\ x,y \in M$.
\item[{\rm  (iv)}] $s$ is an $(R,1_R)$-state-morphism on $M$ if an only if $m_s$ is a state-morphism on $M$.
\end{enumerate}
\end{theorem}

\begin{proof}
First, we characterize Riesz spaces from the assumptions of the theorem:
Since every strictly positive element of $R$ is a strong unit, by \cite[Lem 14.1]{Goo}, $R$ is a simple $\ell$-group, i.e., the only $\ell$-ideals of $R$ are $\{0\}$ and $R$. Therefore, the ideal $\{0\}$ is a unique maximal ideal of $\Gamma(R,1_R)$, so that by Proposition \ref{pr:3.7}(1) and Proposition \ref{pr:3.12}, there is a unique state $\mu$ on $\Gamma(R,1_R)$, it is a state-morphism as well as an extremal state, so that, $\Ker(\mu)=\{0\}$ and $\mu(a)=\mu(b)$ for $a,b \in \Gamma(R,1_R)$ if and only if $a=b$.

Now, let $m$ be a state on $M$. Then the mapping $s(x):=m(x)1_R$, $x \in M$, is trivially an $(R,1_R)$-state on $M$. This is true for each unital Riesz space.

Conversely, let $s$ be an  arbitrary $(R,1_R)$-state on $M$. By the Yosida Theorem \ref{th:Yos}, there is a compact Hausdorff topological space $T\ne \emptyset$ such that the unital Riesz space $(R,1_R)$ can be injectively embedded into the unital Riesz space $(C(T),1_T)$ of continuous functions on $T$ as its Riesz subspace. If $\phi$ is this embedding, then $(\phi(R),1_R)$ is a unital Riesz space whose every strictly positive element is a strong unit for $\phi(R)$.

For any $t\in T$, let us define $s_t:\Gamma(\phi(R),1_T)\to [0,1]$ by $s_t(f):= f(t)$, $f \in \Gamma(\phi(R),1_T)$. Then each $s_t$ is a state-morphism on $\Gamma(\phi(R),1_T)$. Define a mapping $\hat s_t: M \to [0,1]$ as $\hat s_t= s_t \circ \phi\circ s$. Since every strictly positive element of $\phi(R)$ is a strong unit for it, by the above first note from the beginning of our proof, we conclude that $s_t=s_{t'}$ for all $t,t'\in T$. Hence, $\hat s_t=\hat s_{t'}$. Thus we denote by $\hat s=\hat s_{t_0}$ for an arbitrary $t_0\in T$. Consequently, $\hat s(x)$ is a constant function on $T$ for each $x \in M$, and the range of $\hat s$ is a linearly ordered set, therefore, the range of $s(M)$ is a linearly ordered set in $R$.

If we define $m_s(x):=s_{t_0}(\phi(s(x)))$, $x \in M$, then $m_s$ is a state on $M$. Consequently, $s(x):=m_s(x)1_R$, $x \in M$.

Now it is clear that the mapping $\Phi: \mathcal S(M,R,1_R)\to S(M)$ defined by $\Phi(s):= m_s$, $s \in \mathcal S(M,R,1_R)$, is a bijective affine mapping.

Due to this bijective affine mapping $\Phi$, we see that in view of \cite[Prop 4.7]{DvuS}, statements (i)--(iv) are mutually equivalent.
\end{proof}

We note that the results of the precedent theorem are not surprising because if $(R,1_R)$ is a unital Riesz space, whose every strictly positive element is a strong unit for $R$, then $(R,1_R)\cong (\mathbb R,1)$. Indeed, since every $\ell$-ideal of $R$ is also a Riesz ideal of $R$ and vice versa (this is true for each Riesz space), then $\{0\}$ is a unique proper Riesz ideal of $R$, see \cite[Lem 14.1]{Goo}, therefore, it is maximal, and since $R$ is Archimedean, by the proof of the Yosida Representation Theorem, see \cite[Chap 45]{LuZa}, $T$ is in fact the set of maximal Riesz ideals of $(R,1_R)$ which is topologized by the hull-kernel topology and in our case, $T$ is a singleton. Consequently, $(R,1_R)\cong (\mathbb R,1)$, and we can apply \cite[Prop 4.7]{DvuS}.

In the next proposition we show that if every strictly positive element of a unital $\ell$-group $(G,u)$ is a strong unit for $G$ and $(R,1_R)$ is Archimedean, then $M=\Gamma(G,u)$ possesses a unique $(R,1_R)$-state.

\begin{proposition}\label{pr:3.14}
Let $M=\Gamma(G,u)$ and every non-zero element of $M$ be a strong unit for $G$. Then every $(R,1_R)$-state $s$ on $M$ is an $(R,1_R)$-state-morphism for every unital Riesz space $(R,1_R)$, $\Ker(s)$ is a maximal ideal of $M$, $M$ is an MV-algebra, and $\mathcal S(M,R,1_R)\ne \emptyset$. If, in addition, $(R,1_R)$ is an Archimedean unital Riesz space, then $|\mathcal S(M,R,1_R)|=1$.
\end{proposition}

\begin{proof}
The hypotheses imply that $M$ is Archimedean, and by \cite[Thm 4.2]{Dvu1}, $M$ is an MV-algebra. Whence, every strictly positive element of $G$ is a strong unit of $G$, where $(G,u)$ is a unital $\ell$-group such that $M\cong \Gamma(G,u)$. Therefore, $(G,u)$ is an Abelian unital $\ell$-group. By \cite[Lem 14.1]{Goo}, $G$ is a simple $\ell$-group, so that $G$ has a unique proper $\ell$-ideal, namely the zero $\ell$-ideal. Then $\{0\}$ is a unique maximal $\ell$-ideal of $M$. Hence, by Proposition \ref{pr:3.4}, $M$ possesses an $(R,1_R)$-state for every unital Riesz space $(R,1_R)$.

Let $s$ be an $(R,1_R)$-state on $M$. Since $\Ker(s)\ne M$, $\Ker(s)=\{0\}$, and hence $\Ker(s)$ is a maximal ideal of $M$. By Proposition \ref{pr:3.7}(1), $s$ is an $(R,1_R)$-state-morphism on $M$.

Now let $(R,1_R)$ be an Archimedean unital Riesz space. Let $s_1$ and $s_2$ be $(R,1_R)$-states on $M$.  Then $s:=1/2s_1+1/2s_2$ is also an $(R,1_R)$-state on $M$ and in view of the first part of the present proof, $s,s_1,s_2$ are $(R,1_R)$-state-morphisms such that $\Ker(s)=\Ker(s_1)=\Ker(s_2)=\{0\}$ are maximal ideals of $M$, which by Theorem \ref{th:3.11} yields $s_1=s=s_2$. In particular, the unique $(R,1_R)$-state on $M$ is extremal.
\end{proof}

\begin{proposition}\label{pr:3.15}
Let $T$ be a non-void set and $R=C_b(T)$ be the Riesz space of bounded real-valued functions on $T$, and let $1_T$ be the constant function $1_T(t)=1$, $t \in T$. Then $1_R=1_T$ is a strong unit for $R$. Let $s$ be an $(R,1_R)$-state on a pseudo MV-algebra $M$. The following statements are equivalent:
\begin{enumerate}

\item[{\rm (i)}] $s$ is an extremal $(R,1_R)$-state on $M$.

\item[{\rm  (ii)}] $s$ is an $(R,1_R)$-state-morphism on $M$.

\item[{\rm  (iii)}] $s(x\wedge y) =  s(x)\wedge s(y),\ x,y \in M$.
\end{enumerate}
In addition,
\begin{equation}\label{eq:part}
\mathcal S_\partial(M,R,1_R)=\mathcal{SM}(M,R,1_R).
\end{equation}
\end{proposition}

\begin{proof}
(i) $\Rightarrow$ (ii). Let $s$ be an extremal $(R,1_R)$-state
on $M$. If we define an MV-algebra $\Gamma(R,1_R)$, then every mapping $s_t: \Gamma(R,1_R)\to [0,1]$ defined by $s_t(f)=f(t)$, $f \in \Gamma(R,1_R)$, is a state-morphism on $\Gamma(R,1_R)$. We assert that $s_t\circ s$ is an extremal state on $M$ for each $t\in T$. Indeed, let $m_1^t,m_2^t$ be states on $M$ and $\lambda \in (0,1)$ such that $s_t \circ s= \lambda m_1^t +(1-\lambda)m_2^t$. Define a function $m_i: M\to [0,1_R]$ such that $(m_i(x))(t)=m_i^t(x)$ for $t \in T$ where $x \in M$ and $i=1,2$. Then $m_1(x),m_2(x) \in C_b(T)$ for each $x \in M$, so that $m_1,m_2$ are $(R,1_R)$-states on $M$ such that $s=\lambda m_1 + (1-\lambda)m_2$. Since $s$ is extremal, $s=s_1=s_2$ which gives $s_t\circ s= m_1^t=m_2^t$ for each $t\in T$, and $s_t \circ s$ is an extremal state on $M$.

By \cite[Prop 4.7]{DvuS}, $s_t\circ s$ is a state-morphism on $M$, therefore,  $s_t(s(x\wedge y))= \min\{s_t(s(x)),s_t(s(y))\}$, that is $s(x\wedge y) =s(x)\wedge s(y)$ for all $x,y \in M$ and $s$ is an $(R,1_R)$-state morphism, see Proposition \ref{pr:3.6}.

(ii) $\Rightarrow$ (i). Let $s$ be an $(R,1_R)$-state-morphism on $M$ and let $s_1,s_2$ be $(R,1_R)$-states on $M$ such that $s = \lambda s_1 +(1-\lambda)s_2$ for some $\lambda \in (0,1)$. Let $s_t$ be a state-morphism from the foregoing implication for every $t \in T$. Then $s_t\circ s= \lambda s_t\circ s_1 +(1-\lambda)s_t\circ s_2$. Applying \cite[Prop 4.7]{DvuS}, we have $s_t \circ s$ is an extremal state because $s_t\circ s$ is a state-morphism on $M$. Therefore, $s_t\circ s= s_t \circ s_1 = s_t \circ s_2$ for each $t\in T$ which in other words means that $s=s_1=s_2$, that is, $s$ is an extremal $(R,1_R)$-state on $M$.

The equivalence of (ii) and (iii) was established in Proposition \ref{pr:3.6}. Equation (\ref{eq:part}) follows from the equivalence of (i) and (ii).
\end{proof}

\begin{corollary}\label{co:3.16}
Let $(R,1_R)=(\mathbb R^n,1_{\mathbb R^n})$, $n\ge 1$, where $1_\mathbb R =(1,\ldots,1)$, and let $s$ be an $(R,1_R)$-state on a pseudo MV-algebra $M$. The following statements are equivalent:
\begin{enumerate}

\item[{\rm (i)}] $s$ is an extremal $(R,1_R)$-state on $M$.

\item[{\rm  (ii)}] $s$ is an $(R,1_R)$-state-morphism on $M$.

\item[{\rm  (iii)}] $s(x\wedge y) =  s(x)\wedge s(y),\ x,y \in M$.
\end{enumerate}
Moreover, {\rm (\ref{eq:part})} holds.
\end{corollary}

\begin{proof}
It follows from Proposition \ref{pr:3.15} because $(\mathbb R^n,1_{\mathbb R^n})\cong (C_b(T),1_T)$, where $|T|=n$.
\end{proof}

The latter result extends a characterization of state-morphisms and extremal states from \cite[Prop 4.7]{DvuS}, because if $T$ is a singleton, then $(C(T),1_T)$ corresponds in fact to $(\mathbb R,1)$.

\begin{theorem}\label{th:3.17}
Let $(R,1_R)$ be an Archimedean unital Riesz space, $M$ a pseudo MV-algebra, and $s,s_1,s_2$ be $(R,1_R)$-states on $M$.

{\rm (1)} If $s$ is an $(R,1_R)$-state-morphism on $M$, then $s$ is an extremal $(R,1_R)$-state.

{\rm(2)} If $s$ is an $(R,1_R)$-state such that $\Ker(s)$ is a maximal ideal, then $s$ is an $(R,1_R)$-state-morphism and an extremal $(R,1_R)$-state on $M$ as well.

{\rm (3)} Let $s_1, s_2$ be  $(R,1_R)$-states on $M$ such that $\Ker(s_1)=
\Ker(s_2)$ and $\Ker(s_1)$ is a maximal ideal. Then $s_1$ and $s_2$ are $(R,1_R)$-state-morphisms and extremal $(R,1_R)$-states such that $s_1=s_2$.

{\rm (4)} Let $s$ be an $(R,1_R)$-state on $M$ such that $M_s=M/\Ker(s)$ is linearly ordered. Then $s$ is an $(R,1_R)$-state-morphism and an extremal $(R,1_R)$-state, and $\Ker(s)$ is a maximal ideal of $M$.
\end{theorem}

\begin{proof}
(1) Due to the Yosida Representation Theorem \ref{th:Yos}, there is a compact Hausdorff topological space $T \ne \emptyset$ such that $(R,1_R)$ can be embedded into $(C(T),1_T)$ as its Riesz subspace; let $\phi$ be the embedding. For each $t\in T$, the function $s_t:\Gamma(C(T),1_T)\to [0,1]$ defined by $s_t(f):=f(t)$, $f \in C(T)$, is a state-morphism on $\Gamma(C(T),1_T)$ for each $t \in T$. Then the mapping $m_t:=s_t\circ \phi\circ s$ is a state-morphism on $M$, so that by \cite[Prop 4.7]{DvuS}, $m_t$ is an extremal state on $M$.  Let $s=\lambda s_1 +(1-\lambda)s_2$, where $s_1,s_2$ are $(R,1_R)$-states on $M$ and $\lambda \in (0,1)$.
Then
$$
m_t= s_t \circ \phi\circ s = \lambda s_t\circ \phi\circ s_1 +(1-\lambda)\circ \phi s_2,
$$
which implies $m_t= s_t\circ \phi\circ s_1=s_t\circ \phi\circ s_2$ for each $t \in T$. Hence, $\phi (s(x))=\phi(s_1(x))=\phi(s_2(x))$ for every $x \in M$, that is $s(x)=s_1(x)=s_2(x)$, and finally $s=s_1 =s_2$.

(2) Let $s$ be an $(R,1_R)$-state on $M$ such that $\Ker(s)$ is a maximal ideal. By Proposition \ref{pr:3.7}, $s$ is an $(R,1_R)$-state-morphism which by the first part of the present proof entails $s$ is extremal.

(3) Let $s_1, s_2$ be  $(R,1_R)$-states on $M$ such that $\Ker(s_1)=
\Ker(s_2)$ and $\Ker(s_1)$ is a maximal ideal. By (2), $s_1$ and $s_2$ are extremal, and by Proposition \ref{pr:3.7}, $s_1$ and $s_2$ are also $(R,1_R)$-state-morphisms on $M$, so that Theorem \ref{th:3.11} implies $s_1=s_2$.

(4) Assume $M_s$ is linearly ordered. By Proposition \ref{pr:3.2}(xii), $M_s$ is an MV-subalgebra of the standard MV-algebra $\Gamma(\mathbb R,1)$ of the real line $R$. Then, for the $(R,1_R)$-state $\tilde s$ on $M/\Ker(s)=M_s$ induced by $s$, we have $\Ker(\tilde s)=\{0\}$ is a maximal ideal of $M_s$ and $\Ker(s)$ is a maximal ideal of $M$. Hence, there is a subgroup $\mathbb R_0$ of $\mathbb R$ such that $1 \in \mathbb R_0$ and $M_s \cong \Gamma(\mathbb R_0,1)$. Then there is a unique state-morphism (= extremal state in this case) $m$ on $M_s$. The mapping $s_0(x):= m(x)1_R$ for each $x \in M$ is an $(R,1_R)$-state-morphism on $M$. Since $\mathbb R_0$ is a simple $\ell$-group, applying Proposition \ref{pr:3.14}, we have $s=s_0$ and $s$ is an $(R,1_R)$-state-morphism and an extremal $(R,1_R)$-state on $M$, as well.
\end{proof}

By \cite[Prop 4.3]{DvuS}, a state $s$ on a pseudo MV-algebra is a state-morphism iff $\Ker(s)$ is a maximal ideal. In  the following example and proposition we show that it can happen that an $(R,1_R)$-state-morphism $s$, consequently an extremal $(R,1_R)$-state, has the kernel $\Ker(s)$ that is not maximal even for an Archimedean Riesz space $(R,1_R)$. We note that in Example \ref{ex:3.8} we had an analogous counterexample for a non-Archimedean Riesz space.

\begin{example}\label{ex:3.18}
There are a pseudo MV-algebra $M$, an Archimedean unital Riesz space $(R,1_R)$, and an $(R,1_R)$-state-morphism $s$ on $M$, consequently an extremal $(R,1_R)$-state, such that $\Ker(s)$ is not maximal.
\end{example}

\begin{proof}
Let $(R,1_R)=(\mathbb R^n,1_{\mathbb R^n})$, $M=\Gamma(\mathbb R^n,1_{\mathbb R^n})$ and $s:M \to [0,1_{\mathbb R^n}]$ be such that $s(x)=x$, $x \in M$. Then $s$ is an $(R,1_R)$-state-morphism, and by Theorem \ref{th:3.17}(1) it is an extremal $(R,1_R)$-state. Because $\Ker(s)=\{0\}$, then $\Ker(s)$ is a maximal ideal of $M$ iff $n=1$.
\end{proof}

According to \cite{GeIo}, we say that an element $e$ of a pseudo MV-algebra $M$ is {\it Boolean} if $e\wedge e^-=0$, equivalently, $e\wedge e^\sim =0$, equivalently $e\oplus e= e$. Let $B(M)$ be the set of Boolean elements, then $B(M)$ is a Boolean algebra that is an MV-algebra and subalgebra of $M$, and $e^-=e^\sim$; we put $e'=e^-$.  If $s$ is an $(R,1_R)$-state-morphism and $e$ is a Boolean element of $M$, then $s(e)$ is a Boolean element of the MV-algebra $\Gamma(R,1_R)$. We recall that for a sequence $(a_i)_{i=1}^n$ of Boolean elements of $M$ we have that it is summable iff $a_i\wedge a_j=0$ for $i\ne j$; in such a case, $a_1+\cdots +a_n = a_1 \vee \cdots \vee a_n$. If we say that for an ordered finite system $(a_1,\ldots,a_n)$ of Boolean elements of $M$ we assume that $a_1+\cdots +a_n=1$, it can happen that some of $a_i$ are zeros.

We exhibit the latter example in more details. We note that the MV-algebra $\Gamma(\mathbb R^n,1_{\mathbb R^n})$ has each ideal $I$ of $M$ of the form $I=I_1\times\cdots\times I_n$, where $I_i\in \{\{0\},\mathbb R\}$ for each $i=1,\ldots,n$. In particular, all maximal ideals are of the form $I_1\times \cdots\times I_n$ where exactly one $I_i=\{0\}$ and $I_j=\mathbb R$ for $j\ne i$.

\begin{proposition}\label{pr:3.19}
Let $M=\Gamma(\mathbb R^n,1_{\mathbb R^n})$, $(R,1_R)= (\mathbb R^n,1_{\mathbb R^n})$, and let $B(M)$ be the set of Boolean elements of $M$.

{\rm (1)} Let $(a_1,\ldots,a_n)$ be an $n$-tuple of summable elements of $B(M)$ such $\sum_{i=1}^n a_i= 1$. Then the mapping
\begin{equation}\label{eq:R-state0}
s(x)=x_1a_1+\cdots +x_na_n, \quad x=(x_1,\ldots,x_n)\in M,
\end{equation}
is both an $(\mathbb R^n,1_{\mathbb R^n})$-state-morphism and an extremal $(\mathbb R^n,1_{\mathbb R^n})$-state on $M$ as well. Conversely, each $(\mathbb R^n,1_{\mathbb R^n})$-state-morphism on $M$ can be obtained in this way.

{\rm (2)} If $\sigma$ is an arbitrary  permutation of the set $\{1,\ldots,n\}$, then $s_\sigma(x_1,\ldots,x_n)= (x_{\sigma(1)},\ldots,x_{\sigma(n)})$, $x=(x_1,\ldots,x_n)\in M$, is an $(R,1_R)$-state-morphism such that $\Ker(s_\sigma)=\{0\}$.
If $s_0(x)=x$, $x \in M$, $s_0$ is an $(R,1_R)$-state-morphism on $M$ corresponding to the identical permutation. Conversely, every $(\mathbb R^n,1_{\mathbb R^n})$-state-morphism $s$ on $M$ such that $\Ker(s)=\{0\}$ can be obtained in this way.

In addition,
$$
\mathcal S(M,\mathbb R, 1_\mathbb R)=\mathcal S_\partial (M,\mathbb R,1_\mathbb R)= \{s_0\}=\mathcal{SM}(M,\mathbb R,1_\mathbb R).$$

{\rm (3)} For each $i=1,\ldots,n$, let $\pi_i$ be the $i$-th projection from $M$ onto $[0,1]$. Then each $s_i$, where $s_i(x)=\pi_i(x)1_{\mathbb R^n}$, $x \in M$, is an $(\mathbb R^n,1_{\mathbb R^n})$-state-morphism such that $\Ker(s_i)$ is a maximal ideal of $M$, and conversely, every $(\mathbb R^n,1_{\mathbb R^n})$-state-morphism on $M$ whose kernel is a maximal ideal of $M$ is of this form.

{\rm (4)} $$|\mathcal S_\partial(M,\mathbb R^n,1_{\mathbb R^n})|=n^n=| \mathcal {SM}(M,\mathbb R^n,1_{\mathbb R^n})|$$
and
$$\mathcal S(\mathbb R^n,1_{\mathbb R^n})=\Con(\mathcal S_\partial(\mathbb R^n,1_{\mathbb R^n}))=\Con(\mathcal {SM}(\mathbb R^n,1_{\mathbb R^n})),$$
where $\Con$ denotes the convex hull.

\end{proposition}

\begin{proof}
Since the Archimedean unital Riesz space $(\mathbb R^n,1_{\mathbb R^n})$ satisfies the conditions of Corollary \ref{co:3.16}, we have that $\mathcal S_\partial (M,\mathbb R^n,1_{\mathbb R^n})= \mathcal{SM}(M,\mathbb R^n,1_{\mathbb R^n})$. For each $i=1,\ldots,n$, let $e_i$ be the vector of $M$ whose all coordinates are zeros, only at the $i$-th place there is $1$. The elements $e_1,\ldots,e_n$ are unique atoms of the Boolean algebra $B(M)$ and $e_1+\cdots + e_n =1$. The set of Boolean elements of $M$ has $2^n$ elements, and each Boolean element of $M$ is a vector $e\in M$ whose coordinates are only $0$ and $1$. Let $s_0(x)=x$, $x \in M$. If $n=1$, then $\mathcal S_\partial(M,\mathbb R,1_{\mathbb R})=\{s_0\}$.

(1) Let $(a_1,\ldots,a_n)$ be a summable sequence of Boolean elements of $M$
such that $a_1+\cdots +a_n=1$. We define a mapping $s$ by (\ref{eq:R-state0}). Then it is clear that $s$ is an $(\mathbb R^n,1_{\mathbb R^n})$-state on $M$. To show that $s$ is also an $(\mathbb R^n,1_{\mathbb R^n})$-state-morphism, we verify the criterion (iii) of Corollary \ref{co:3.16}. Let $x=(x_1,\ldots,x_n)\in M$ and $y=(y_1,\ldots,y_n)\in M$. Then $x\wedge y = (x_1\wedge y_1,\ldots, x_n\wedge y_n)$. Therefore,
\begin{eqnarray*}
s(x)\wedge s(y) &=& (x_1a_1+\cdots +x_na_n)\wedge (y_1a_1+\cdots +y_na_n)\\
&=& (x_1a_1\vee \cdots \vee x_na_n)\wedge (y_1a_1\vee\cdots \vee y_na_n)\\
&=& (x_1\wedge y_1)a_1 + \cdots +(x_n\wedge y_n)a_n\\
&=& s(x \wedge y),
\end{eqnarray*}
when we have used the fact that in $\ell$-groups if, for $a,b \in G^+$,  $a\wedge b=0$, then $a+b=a\vee b$. Consequently, $s$ is an $(\mathbb R^n,1_{\mathbb R^n})$-state-morphism on $M$.

Conversely, let $n\ge 2$ and let $s$ be an $(\mathbb R^n,1_{\mathbb R^n})$-state-morphism on $M$.
We define $f^s_i=s(e_i)$. For $i=1$, we have $s(1,0,\ldots,0)= s(n/n,0,\ldots,0)=ns(1/n,0,\ldots,0)$ so that $f^s_1/n= s(1/n,0,\ldots,0)$. Now let $0\le m\le n$. Then $s(m/n,0,\ldots,0)=ms(1/n,0,\ldots,0)= \frac{m}{n}f^s_1$. Now let $t$ be a real number from $(0,1)$. Passing to two monotone sequences of rational numbers $\{p_n\}\nearrow t$ and $\{q_n\}\searrow t$, we obtain $p_n s(1,0,\ldots,0)\le s(t,0,\ldots,0)\le q_ns(1,0,\ldots,0)$, so that $s(t,0,\ldots,0)=tf^s_1$. The same is true for each $i$, i.e. $s(0,\ldots,t,\ldots,0)=tf^s_i$. Hence, for each $x=(x_1,\ldots,x_n)\in M$, we have
\begin{equation}\label{eq:R-state}
s(x)=\sum_{i=1}^n x_if^s_i.
\end{equation}
From (\ref{eq:R-state}) we see that if we put $a_i=f^s_i$ for each $i=1,\ldots,n$, then we obtain formula (\ref{eq:R-state0}).

(2) Assume $\Ker(s)=\{0\}$. Then each $f^s_i\ne 0$ and all $f^s_i$'s are mutually different. Since $1_{\mathbb R^n}=f^s_1+\cdots+f^s_n$, we assert that each $f^s_i\in \{e_1,\ldots,e_n\}$. Indeed, if some $f^s_i$ has two coordinates equal $1$, then one of $f^s_j$ for $j\ne i$ has to be zero. Therefore, for each $f^s_i$, there is a unique $e_{j_i}$ such that $f^s_i=e_{j_i}$. This defines a permutation $\sigma$ such that $s=s_\sigma$. Conversely, each $s_\sigma$ is an $(\mathbb R^n,1_{\mathbb R^n})$-state-morphism, consequently an extremal $(\mathbb R^n,1_{\mathbb R^n})$-state whose kernel is $\{0\}$.

(3) Now we exhibit $(\mathbb R^n,1_{\mathbb R^n})$-state-morphisms $s$ whose kernel is a maximal ideal of $M$. It is easy to verify that every $s_1,\ldots,s_n$ defined in the proposition is an $(\mathbb R^n,1_{\mathbb R^n})$-state-morphism on $M$ whose kernel is a maximal ideal. Conversely, let $s$ be an $(\mathbb R^n,1_{\mathbb R^n})$-state-morphism on $M$ such that $\Ker(s)$ is a maximal ideal. From (\ref{eq:R-state0}) we conclude, that there is a unique $i=1,\ldots,n$ such that $a_i=1$ and $a_j=0$ for each $j\ne i$ which entails the result.

(4) Let $\tau: \{1,\ldots,n\}\to \{1,\ldots,n\}$, i.e. $\tau \in \{1,\ldots,n\}^{\{1,\ldots,n\}}$. Then the mapping $s_\tau:M \to [0,1_{\mathbb R^n}]$ defined by $s_\tau(x_1,\ldots,x_n)=(x_{\tau(1)},\ldots, x_{\tau(n)})$, $x=(x_1,\ldots,x_n) \in M$, is an $(\mathbb R^n,1_{\mathbb R^n})$-state-morphism on $M$. Since every $a_i$ for $i=1,\ldots,n$ is a finite sum of some Boolean elements of $e_1,\ldots,e_n$, using (\ref{eq:R-state0}), we see that every $(\mathbb R^n,1_{\mathbb R^n})$-state-morphism on $M$ is of this form.

Finally, we say that a net $\{s_\alpha\}_\alpha$ of $(\mathbb R^n,1_{\mathbb R^n})$-states on $M$ converges weakly to an $(\mathbb R^n,1_{\mathbb R^n})$-state $s$, and we write $\{s_\alpha\}_\alpha\stackrel{w} \rightarrow s$, if, for each $i=1,\ldots,n$,  $\lim_\alpha\pi_i(s_\alpha(x))=\pi_i(s(x))$ for each $x\in M$. Since $\pi_i\circ s$ is in fact a state on $M$, it is easy to see that we have the weak convergence of states on $M$ which gives a compact Hausdorff topology on the state space of $M$.  Whence if, for some net $\{s_\alpha\}_\alpha$ of $(\mathbb R^n,1_{\mathbb R^n})$-states, we have that there is, for each $i=1,\ldots,n$,  $\lim_\alpha \pi_i(s_\alpha(x))=s_i(x)$, $x \in M$, then every $s_i$ is a state on $M$, so that $s(x):=(s_1(x),\ldots,s_n(x))$, $x \in M$, is an $(\mathbb R^n,1_{\mathbb R^n})$-state such that $\{s_\alpha\}_\alpha\stackrel{w} \rightarrow s$. Consequently, the space $\mathcal S(M,\mathbb R^n,1_{\mathbb R^n})$ is a non-void convex compact Hausdorff space, so that by the Krein--Mil'man theorem, \cite[Thm 5.17]{Goo}, we see that every $(\mathbb R^n,1_{\mathbb R^n})$-state lies in the weak closure of the convex hull of $\mathcal S_\partial(M,\mathbb R^n,1_{\mathbb R^n})=\mathcal{SM}(M,\mathbb R^n,1_{\mathbb R^n})$. Since the space $\mathcal{SM}(M,\mathbb R^n,1_{\mathbb R^n})$ has exactly $n^n$ elements, let $\mathcal{SM}(M,\mathbb R^n,1_{\mathbb R^n})=\{s_1,\ldots,s_{n^n}\}$, so that every element $s$ of $\Con(\mathcal{SM}(M,\mathbb R^n,1_{\mathbb R^n}))$ is of the form $s =
\lambda_1 s_1+\cdots+\lambda_{n^n} s_{n^n}$, where each $\lambda_j \in [0,1]$ and $\lambda_1 +\cdots + \lambda_{n^n}$. Hence, there is a net $\{\lambda_1^\alpha s_1+\cdots+\lambda_{n^n}^\alpha\}_\alpha$ from the convex hull which converges weakly to the $(\mathbb R^n,1_{\mathbb R^n})$-state $s$. In addition, for each $i=1,\ldots,n^n$, there is a subnet $\{\lambda_i^{\alpha_\beta}\}_\beta$ of the net $\{\lambda_i^\alpha\}_\alpha$ such that $\lim_\beta \lambda_i^{\alpha_\beta}=\lambda_i$. Whence $\lambda_1+\cdots+\lambda_{n^n}=1$ and $s = \lambda_1 s_1+\cdots +\lambda_{n^n}\lambda_{n^n}$ which finishes the proof.
\end{proof}

A more general type of the weak convergence of $(R,1_R)$-states for a Dedekind $\sigma$-complete Riesz space will be studied in Proposition \ref{pr:conv2} below.

The latter proposition can be extended for $(\mathbb R^m,1_{\mathbb R^m})$-state-morphisms on the MV-algebra $M_n=\Gamma(\mathbb R^n,1_{\mathbb R^n})$ for all integers $m,n\ge 1$.

\begin{proposition}
Let $n,m\ge 1$ be integers and let $(a_1,\ldots,a_n)$ be a summable sequence of Boolean elements from $M_m=\Gamma(\mathbb R^m,1_{\mathbb R^m})$. Then the mapping $s(x):=x_1a_1+\cdots + x_na_n$, $x=(x_1,\ldots,x_n)\in M_n$, is an $(\mathbb R^m,1_{\mathbb R^m})$-state-morphism on $M_n$, and conversely, each $(\mathbb R^m,1_{\mathbb R^m})$-state-morphism on $M_n$ can be described in this way.

Equivalently, let $\tau$ be any mapping from $\{1,\ldots,m\}$ into $\{1,\ldots,n\}$. Then the mapping $s_\tau(x_1,\ldots,x_n)=(x_{\tau(1)},\ldots,x_{\tau(m)})$, $x=(x_1,\ldots,x_n)\in M_n$, is is an $(\mathbb R^m,1_{\mathbb R^m})$-state-morphism on $M_n$, and conversely, each $(\mathbb R^m,1_{\mathbb R^m})$-state-morphism on $M_n$ can be described in this way.

In particular, $|\mathcal{SM}(M_n,\mathbb R^m,1_{\mathbb R^m})|=n^m=|\mathcal S_\partial (M_n,\mathbb R^m,1_{\mathbb R^m})|$.

An $(\mathbb R^m,1_{\mathbb R^m})$-state-morphism $s$ on $M_n$ has maximal kernel if and only if there is $i=1,\ldots,n$ such that $s(x)=\pi_i(x)1_{\mathbb R^m}$, $x\in M_n$, where $\pi_i: \mathbb R^n \to \mathbb R$ is the $i$-th projection.

\end{proposition}

\begin{proof}
The proof follows methods of the proof of Proposition \ref{pr:3.19}.
\end{proof}

Now we present a criterion for $(R,1_R)$ when the kernel of every $(R,1_R)$-state-morphism on an arbitrary pseudo MV-algebra is a maximal ideal.

\begin{proposition}\label{pr:3.20}
Let $(R,1_R)$ be a unital Riesz space. Then every $(R,1_R)$-state-morphism on an arbitrary pseudo MV-algebra has the kernel a maximal ideal if and only if $(R,1_R)$ is isomorphic to $(\mathbb R,1)$.
\end{proposition}

\begin{proof}
Let $(R,1_R)\cong (\mathbb R,1)$ and let $M$ be an arbitrary pseudo MV-algebra. Let $s$ be an $(R,1_R)$-state-morphism on $M$. According to \cite[Prop 4.3]{DvuS}, $\Ker(s)$ is a maximal ideal of $M$.

Conversely, let $M$ be an arbitrary pseudo MV-algebra with an $(R,1_R)$-state-morphism $s$ such that $\Ker(s)$ is a maximal ideal of $M$.  Take the special MV-algebra $M=\Gamma(R,1_R)$ and let $s(x)=x$, $x \in M$. Then clearly $s$ is an $(R,1_R)$-state-morphism with $\Ker(s)=\{0\}$, and by the assumption, $\Ker(s)=\{0\}$ is a maximal ideal of $M$. Therefore, $M$ has only the zero ideal and $M$. Due to the categorical equivalence of Abelian unital $\ell$-groups and MV-algebras, the Riesz space $(R,1_R)$ has only two $\ell$-ideals, $\{0\}$ and $R$ which by \cite[Lem 14.1]{Goo} yields that every strictly positive element of $R$ is a strong unit of $R$.
As it was shown just after Theorem \ref{th:3.13}, this means that $(R,1_R)\cong (\mathbb R,1)$.

The same result we obtain if take into account that the $\ell$-ideal $\{0\}$ of $R$ is a maximal $\ell$-ideal of $R$ generated by $\Ker(s)$. Therefore,
$R\cong R/\{0\}$ and the quotient unital Riesz space $(R/\{0\},1_R/\{0\})$ can be identify with a unital subgroup $(\mathbb R_0,1)$ of $(\mathbb R,1)$. Since $\alpha =\alpha 1\in \mathbb R_0$ for each real number $\alpha$, we conclude $\mathbb R_0=\mathbb R$.
\end{proof}

Now we present another criterion of maximality of $\Ker(s)$ for an $(R,1_R)$-state-morphism $s$ when $(R,1_R)$ is an Archimedean unital Riesz space.

\begin{proposition}\label{pr:3.21}
Let $T$ be a non-void compact Hausdorff space and $M$ be a pseudo MV-algebra. Then the kernel of a $(C(T),1_T)$-state-morphism $s$ on $M$ is a maximal ideal if and only if there is a state-morphism $s_0$ on $M$ such that $s(x)=s_0(x)1_T$, $x \in M$.

The same statement holds for a $(C_b(T),1_T)$-state-morphism on $M$.
\end{proposition}

\begin{proof}
Let $s$ be a $(C(T),1_T)$-state-morphism on $M$, and for each $t\in T$, let $s_t:M \to [0,1]$ be a mapping given by $s_t(x):=s(x)(t)$, $x \in M$. Then each $s_t$ is a state-morphism on $M$, and $\Ker(s)=\bigcap_{t \in T}\Ker(s_t)$. Hence, if $\Ker(s)$ is maximal, then from $\Ker(s)\subseteq \Ker(s_t)$ we conclude $\Ker(s)=\Ker(s_t)$ because every $\Ker(s_t)$ is a maximal ideal of $M$, see \cite[Prop 4.3]{DvuS}, so that $s_t = s_{t'}$ for all $t,t'\in T$ which gives the desired result.

The converse statement is evident.
\end{proof}

Using \cite[Lem 8.10]{Goo}, it is possible to show that if $s$ is an extremal state on an MV-algebra $M$, then for each Boolean element $e\in B(M)$, we have $s(e)\in \{0,1\}$. Using the Proposition \ref{pr:3.19}, we can show that this is not true for each extremal $(R,1_R)$-state.

\begin{example}\label{ex:3.22}
There are an MV-algebra $M$, an Archimedean unital Riesz space $(R,1_R)$, an $(R,1_R)$-state-morphism, i.e. an extremal $(R,1_R)$-state-morphism $s$ and a Boolean element $e \in B(M)$ such that $s(e)\notin \{0,1_R\}$.
\end{example}

\begin{proof}
Take a particular case of Proposition \ref{pr:3.19}, namely $(R,1_R)=(\mathbb R^3,1_{\mathbb R^3})$, $M= \Gamma(\mathbb R^3,1_{\mathbb R^3})$ and let $s$ be an $(\mathbb R^3,1_{\mathbb R^3})$-state-morphism given by (\ref{eq:R-state0}), where $a_1=(1,0,0)$, $a_2=(0,1,0)$, $e_3=(0,0,1)$. If we take the Boolean element $e=(1,0,0)$, then $s(e)=a_1 \notin \{0,(1,1,1)\}$.
\end{proof}

\section{Metrical Completion of Unital $\ell$-groups with Respect  $(R,1_R)$-States}

In this section, we show that if $(R,1_R)$ is a unital Archimedean or even a Dedekind complete Riesz space and $s$ is an $(R,1_R)$-state on a unital $\ell$-group $(G,u)$, then $G$ can be metrically completed with respect to a norm induced by $s$. If $R$ is Dedekind complete, in particular, if  $(R,1_R)=(\mathbb R^n,1_{\mathbb R^n})$, where $1_{\mathbb R^n}=(1,\ldots,1)$, then the metrical completion of $G$ gives a Dedekind complete $\ell$-group.

Let $T$ be a compact Hausdorff topological space. We endow $C(T)$ with the uniform topology generated by the norm $\|\cdot\|_T$, i.e. $\|f\|_T = \sup\{|f(t)| \colon t \in T\}$.

Let $(R,1_R)$ be an Archimedean unital Riesz space. By the Yosida Representation Theorem, Theorem \ref{th:Yos}, there is a compact Hausdorff topological space $T$ such that $(R,1_R)$ can be embedded into $(C(T),1_T)$ as its Riesz subspace. If $\phi$ is this embedding, then the image $\phi(R)$ is uniformly dense in $C(T)$. More precisely, let $T$ be the set of maximal ideals of $(R,1_R)$. If we define the hull-kernel topology on $T$, $T$ becomes a non-void compact Hausdorff topological space. If $I$ is a maximal ideal of $(R,1_R)$, then the quotient $R/I$ can be identify with a unital subgroup $(\mathbb R_0,1)$ of $(\mathbb R,1)$ (in fact, $\mathbb R_0=\mathbb R$). Moreover, given $x \in R$, the mapping $\widehat x:T \to \mathbb R$ given by $\widehat x(I):=x/I$, $I \in T$, is a continuous function on $T$, and the mapping
$\phi:R \to C(T)$ defined by $\phi(x)=\widehat x$, $x\in R$, is by \cite[Thm 45.3]{LuZa} an isomorphic embedding of the unital Archimedean space $(R,1_R)$ into $(C(T),1_T)$ such that $\phi(R)$ is uniformly dense in $C(T)$. We call this embedding the {\it canonical embedding}, the triple $(C(T),1_T,\phi)$ is said to be the {\it canonical representation} of $(R,1_R)$, and we shall write $(R,1_R)\sim (C(T),1_T,\phi)$. We note that according to \cite[Thm 45.4]{LuZa} if, in addition, $(R,1_R)$ is Dedekind $\sigma$-complete, then $\phi(R)=C(T)$, and $\phi$ is an isomorphism of Riesz spaces.

Thus, let $(R,1_R)$ be an Archimedean unital Riesz space with the canonical representation $(C(T),1_T,\phi)$. Let $s$ be an $(R,1_R)$-state on a pseudo MV-algebra $M$. Assume $M=\Gamma(G,u)$ for some unital $\ell$-group $(G,u)$. Due to the categorical equivalence,  $s$ can be extended to a unique $(R,1_R)$-state on $(G,u)$; we denote it by $\hat s$. Due to (ix) and (xii) of Proposition \ref{pr:3.2}, we have that $M_s:=M/\Ker(s)$ is an Archimidean MV-algebra, and $M_s\cong \Gamma(G_s,u/I_s)$, where $G_s = G/I_s$ and $I_s$ is an $\ell$-ideal of $G$ generated by $\Ker(s)$. Then the $(R,1_R)$-state $\tilde s$ on $M_s$ can be uniquely extended to a unique $(R,1_R)$-state $\widehat {\tilde s}$ on $(G_s,u/I_s)$. In addition, $I_s=\{x\in G\colon \hat s(|x|)=0\}$, where $|x|= x^+ + x^-$, $x^+= x\vee 0$ and $x^-=-(x\wedge 0)$. In general, if $s$ is an $(R,1_R)$-state on a unital $\ell$-group $(G,u)$, then the {\it kernel} of $s$ is the set $\Ker(s)=\{x \in G\colon s(|x|)=0\}$.

We define a pseudo norm $|\cdot|_s$ on $M$ as follows
$$
|x|_s:= \| \phi(s(x))\|_T:=\sup\{|\phi(s(x))(t)| \colon t \in T\},\quad x \in M,
$$
This pseudo norm can be extended to a pseudo norm $|\cdot |_s$ on $G$ as follows
$$|x|_s:= \| \phi(\hat s(x))\|_T=\sup\{|\phi(\hat s(x))(t)| \colon t \in T\}, \quad x \in G.
$$
The simple properties of $|\cdot|_s$ are as follows: For $x,y, \in G$, we have
\begin{enumerate}
\item[{\rm (i)}] $|x+y|_s\le |x|_s + |y|_s$.

\item[{\rm (ii)}] $|\phi(\hat s(x))(t)-\phi(\hat s(y))(t)|\le |x-y|_s$, $t \in T$.

\item[{\rm (ii)}] $|nx|_s  \le |n|\cdot |x|_s$, $n \in \mathbb Z$.

\item[{\rm (iv)}] $|-x|_s = |x|_s$, $|0|_s = 0$.

\item[{\rm (v)}] If $y \in G^+$ and $-y \le x \le y$, then $|x|_s\le |y|_s$.
\end{enumerate}

Since our aim is to study extremal $(R,1_R)$-states on $M$ and in view of Lemma \ref{le:3.3}, an $(R,1_R)$-state $s$ on $M$ is extremal iff so is $\tilde s$ on $M_s:=M/\Ker(s)$, without loss of generality we will assume that $M$ is an Archimedean MV-algebra, $(G,u)$ is an Archimedean (Abelian) unital $\ell$-group, and $s(x)=0$ for $x \in M$ iff $x=0$ because $\tilde s([x])=0$ iff $[x]=0$. Then $|\cdot|_s$ is a norm on $M$ and $(G,u)$, respectively, and $d_s(x,y):=|x-y|_s$ defines a metric called the $s$-{\it metric}.

Because of triangle inequality (i), addition and subtraction in $G$ are uniformly continuous with respect to $d_s$. Hence, the $d_s$-completion $\overline G$ of $G$ is a topological Abelian group, and the natural mapping $\psi:G \to \overline G$ is a continuous group homomorphism. We define a relation $\le$ on $\overline G$ so that for any $x,y \in \overline G$, we put $x\le y$ iff $y-x$ lies in the closure of $\psi(G^+)$. We note $\Ker(\psi)=\{x \in G: |x|_s=0\}$. In the following statement we show, in particular, that this relation is a translation-invariant partial order on $\overline G$, that is, given $x,y,z \in \overline G$, $x\le y$ implies $x+z\le y+z$. In what follows, we are inspired by ideas and proofs from \cite[Chap 12, Thm 12.2]{Goo}.

\begin{proposition}\label{pr:4.1}
Let $(R,1_R)$ be an Archimedean unital Riesz space with $(R,1_R)\sim (C(T),1_T,\phi)$. Let $(G,u)$ be an Archimedean unital $\ell$-group, and $s$ an $(R,1_R)$-state on $(G,u)$ such that if $s(x)=0$ for some $x\ge 0$, then $x=0$. Let $\overline G$ be the $d_s$-completion of $G$, $\psi: G \to \overline G$ be the natural embedding, and let $\overline{d}_s$ denote the induced metric on $\overline G$.
Then
\begin{enumerate}
\item[{\rm (i)}] $\overline{G}$ is a directed po-group with positive cone equal to the closure of $\psi(G^+)$.

\item[{\rm (ii)}] There is a unique continuous mapping $\overline s:\overline G\to C(T)$ such that $\phi \circ s = \overline s \circ \psi$, and $\overline s$ is a positive homomorphism of po-groups.

\item[{\rm (iii)}] $\overline d_s(g,0)= \|\overline s(g)\|_T$ for each $g\in \overline G$, and $\overline d_s(g,0)=|g|_s= \|\phi(s(g))\|_T$ for each $g \in G$.
\end{enumerate}
\end{proposition}

\begin{proof}
(i)  Let $C$ be the closure of $\psi(G^+)$ in $\overline G$. Since $\psi(0)=0$, we have $0 \in C$. As $\psi(G^+)$ is closed under addition, the continuity of addition in $\overline G$ entails that $G$ is closed under addition, so that $C$ is a cone.  Now let $x,-x \in C$. Take two sequences $\{x_n\}$ and $\{y_n\}$ in $G^+$ such that $\psi(x_n)\to x$ and $\psi(y_n)\to -x$. Then $\psi(x_n+ y_n)\to 0$.  Since
$$
\overline d_s(\psi(x_n+y_n),0)=d_s(x_n+y_n,0)= |x_n+y_n|_s=\|\phi(s(x_n+y_n))\|_T\to 0
$$
for all $n$, consequently, the sum $\phi(s(x_n))+\phi(s(y_n))$  of positive real-valued continuous functions on $T$ converges uniformly on $T$ to the zero function $0_T$ on $T$. Then $0_T\le \phi(s(x_n))\le \phi(s(x_n+y_n))$ for all $n$. Whence, $\phi(s(x_n))\rightrightarrows 0_T$, and therefore, $x=0$. Thus $C$ is a strict cone, and $\overline G$ becomes a po-group with positive cone $C$.

Now we show that $\overline G$ is directed. Let $x \in \overline G$, and let us choose a sequence $\{x_n\}$ in $G$ such that $|x_{n+1}-x_n|_s< 1/2^n$ for all $n$. Then $x_{n+1}-x_n= a_n - b_n$, where $a_n=(x_{n+1}-x_n)\vee 0\in G^+$ and $b_n= -((x_{n+1}-x_n)\wedge 0)\in G^+$. Then $|x_{n+1}-x_n|_s= |a_n + b_n|_s < 1/2^n$.

Since $a_n\le a_n+b_n$, we have $|a_n|_s \le |a_n+b_n|_s < 1/2^n$ for all $n$. Therefore, the partial sums of the series $\sum_n a_n$ form a Cauchy sequence with respect to $d_s$. Consequently, the series $\sum_n\psi(a_n)$ converges to an element $a \in \overline G$. As the partial sums of this series all lie in $\psi(G^+)$, then $a \in \overline G^+$. In the same way, the series $\sum_n \psi(b_n)$ converges to an element $b \in \overline G^+$. We have
\begin{align*}
a-b&= \lim_k \sum_{n=1}^k \psi(a_n-b_n)= \lim_k \sum_{n=1}^k \psi(x_{n+1}-x_n)\\
&= \lim_k (\psi(x_{k+1})-\psi(x_1))=x-\psi(x_1).
\end{align*}
Since $x_1=c-d$ for some $c,d \in G^+$, we have $x =(a+\psi(c))-(b+\psi(d))$ with $a+\psi(c)$ and $b+\psi(d)\in \overline G^+$ which established $\overline G$ is directed.

(ii) Let $g \in \overline G$. Choose two sequences $\{f_n\}$ and $\{h_n\}$ in $G$ such that $\overline d_s(f_n,g)\to 0$ and $\overline d_s(h_n,g) \to 0$. Then $\phi(s(f_n))\rightrightarrows f \in C(T)$ and $\phi(s(h_n))\rightrightarrows h\in C(T)$. We assert that $f=h$. Indeed
\begin{align*}
\|f-h\|_T&= \|f-\phi(s(f_n))\|_T+ |\phi(s(f_n))-\phi(s(h_n))|_s+ \|\phi(s(h_n))-h\|_T\\
&= \|f-\phi(s(f_n))\|_T+ |f_n- h_n|_s+ \|\phi(s(h_n))-h\|_T\to 0,
\end{align*}
which yields $f=h$.  Therefore, we can define unambiguously $\overline s: \overline G\to C(T)$ as follows: $\overline s(g)= \lim_n \phi(s(g_n))$, $g \in \overline G$, for any sequence $\{g_n\}$ of elements of $G$ such that $\overline d_s(g_n,g)\to 0$.  Then evidently $\overline s(g)\ge 0_T$ if $g\in \overline G^+$, and $\overline s(g+h)=\overline s(g)+\overline s(h)$ for all $f,g \in \overline G$ as well as $\phi \circ s = \overline s \circ \psi$.
In addition, $\overline d_s(g,0)= \|\overline s(g)\|_T$, $g\in \overline G$.

If $\overline d_s(g,g_n)\to 0$ for $g \in \overline G$ and for a sequence $\{g_n\}$ of elements of $\overline G$, then $\|\overline s(g)-\overline s(G_n)\|_T= \overline d_s(g,g_n) \to 0$, so that $\overline s$ is continuous on $\overline G$. If $s':\overline G\to C(T)$ is a continuous positive homomorphism of po-groups such that $\phi \circ s =  s' \circ \psi$, then $s(g)$ and $s'(g)$ coincide for each $g \in G$, so that $s'=\overline s$.

(iii) It follows from the end of the proof of (ii).
\end{proof}

\begin{remark}\label{re:4.2}
{\rm The po-group $\overline G$ from the latter proposition is said to be the {\it metrical completion} of $G$ with respect to an $(R,1_R)$-state $s$. Nevertheless that it was supposed that $G$ is an Archimedean $\ell$-group and $s$ has the property $s(x)=0$ for $x \in G^+$ implies $x$, passing to the $(R,1_R)$-state $\widehat {\tilde s}$ on $G_s$, the metrical completion of $G_s$ from Proposition \ref{pr:4.1} with respect to the $(R,1_R)$-state $\widehat{\tilde s}$ is also in fact a metrical completion of $G$ with respect to the $(R,1_R)$-state $s$. In other words, $G$ can be homomorphically embedded into an Abelian metrically complete po-group $\overline G$.
}
\end{remark}

\begin{proposition}\label{pr:4.3}
Let the conditions of Proposition {\rm \ref{pr:4.1}} hold. If $\{x_\alpha\}_\alpha$ and $\{y_\alpha\}_\alpha$ be nets in $\overline G$ such that $x_\alpha \to x$ and $y_\alpha \to y$. If $x_\alpha \le y_\alpha$ for each $\alpha$, then $x\le y$.
\end{proposition}

\begin{proof}
The differences $y_\alpha - x_\alpha$ form a net in $\overline G$ which converges to $y-x$. As $\overline G^+$ is closed in $\overline G$, we have that $y -x$ lies in $\overline G^+$, and consequently, $x\le y$.
\end{proof}

In what follows, we show that the metrical completion $\overline G$ of $G$ enjoys also some lattice completeness properties. In order to do that, we have to strengthen conditions posed to the Riesz space $(R,1_R)$ assuming $R$ is Dedekind complete. Then due to Theorem \ref{th:sigma}, we have the canonical representation $(R,1_R)\sim (C(T),1_T,\phi)$ and
$(R,1_R)\cong (C(T),1_T)$, where $T\ne \emptyset$ is a compact Hausdorff extremally disconnected topological space. Such a situation is e.g. when $(R,1_R)=(\mathbb R^n,1_{\mathbb R^n})$, $n\ge 1$, then $(R,1_R)\cong (C(T),T)$, where $|T|=n$ and every singleton of $T$ is clopen.

\begin{proposition}\label{pr:4.4}
Let $(R,1_R)$ be a Dedekind complete unital Riesz space with the canonical representation $(R,1_R)\sim (C(T),1_T,\phi)$, where $T\ne \emptyset$ is a Hausdorff compact extremally disconnect topological space. Let $(G,u)$ be an Archimedean unital $\ell$-group and let $s$ be an $(R,1_R)$-state on $(G,u)$ such that if $s(x)=0$ for some $x\ge 0$, then $x=0$. Let $\overline G$ be the $d_s$-completion of $G$, $\psi: G \to \overline G$ be the natural mapping, and let $\overline{d}_s$ denote the induced metric on $\overline G$. Let $\{x_\alpha\colon \alpha \in A\}$ be a net of elements of $\overline G$ which is bounded above and $x_\alpha \le x_\beta$ whenever $\alpha \le \beta$, $\alpha,\beta \in A$.
Then there is an element $x^* \in \overline G$ such that $x_\alpha \to x^*$ and $x^*$ is the supremum of $\{x_\alpha: \overline \alpha \in A\}$ in $\overline G$.
\end{proposition}

\begin{proof}
Since $C(T)$ is Dedekind complete, there is a continuous function $f\in C(T)$ such that $f = \bigvee_\alpha \overline s(x_\alpha)$. Then $f(t)=\sup_\alpha \overline s(x_\alpha)(t)=\lim_\alpha \overline s(x_\alpha)(t)$ for each $t \in T$. Applying the Dini Theorem, see e.g. \cite[p. 239]{Kel}, for the net $\{\overline s(x_\alpha)\colon \alpha \in A\}$ of continuous functions on $T$, the net converges uniformly to $f$. Consequently, $\{x_\alpha\colon \alpha \in A\}$ is a Cauchy net in $\overline G$, so that it converges to some $x^* \in \overline G$.

For any $x_\alpha$, the subnet $\{x_\beta\colon x_\beta \ge x_\alpha\}$ also converges to $x^*$, whence by Proposition \ref{pr:4.3}, $x^*\ge x_\alpha$. Now, let $z \in \overline G$ be any upper bound for $\{x_\alpha \colon \alpha \in A\}$. Applying again Proposition \ref{pr:4.3}, we conclude $x^*\le z$, which proves that $x^*$ is the supremum in question.
\end{proof}


\begin{proposition}\label{pr:4.5}
Let the conditions of Proposition {\rm \ref{pr:4.4}} hold. Then $\overline G$ has interpolation.
\end{proposition}

\begin{proof}
We show that $\overline G$ has interpolation, that is, if $x_1,x_2 \le y_1,y_2$ for $x_1,x_2,y_1,y_2 \in\overline G$, there is a $z \in \overline G$ such that $x_1,x_2 \le z \le y_1,y_2$. To prove this we follow ideas of the proof of \cite[Thm 12.7]{Goo}.

There are four sequences of elements of $G$, $\{x_{in}\}_n$, $\{y_{in}\}_n$ for $i=1,2$, such that
$$\overline d_s(\psi(x_{in}),x_i)< 1/2^{n+5} \quad \text{and} \quad
\overline d_s(\psi(y_{jn}),y_j)< 1/2^{n+5}
$$
for all $i,j =1,2$ and all $n$. For all $i,n$, we have
\begin{align*}
|x_{i,n+1}-x_{in}|_s &= \overline d_s(\psi(x_{i,n+1}),\psi(x_{in})) \le
\overline d_s(\psi(x_{i,n+1}),x_i)+ \overline d_s(x_i,\psi(x_{in}))\\
&< 1/2^{n+6}+1/2^{n+5} < 1/2^{n+4}.
\end{align*}
Similarly, $|y_{j,n+1}-y_{jn}|_s<1/2^{n+4}$. We shall  construct Cauchy sequences $\{b_n\}$ and $\{z_n\}$ of elements in $G$ such that $b_n \to 0$ and $x_{in}\le z_n\le y_{jn}+b_n$ for all $i,j,n$. The limit $\{\psi(z_n)\}$ provides an element in $\overline G$ to interpolate between $x_1,x_2$ and $y_1,y_2$.

We first construct elements $a_1,a_2,\ldots $ in $G$ such that $|a_n|_s< 1/2^{n+2}$ for all $n$ and also
$$
x_{in}-a_n\le x_{i,n+1}\le x_{in}+a_n \quad \text{and} \quad
y_{jn}-a_n\le y_{j,n+1}\le y_{jn}+a_n
$$
for all $i,j,n$.

For each $i,n$, we have $\overline d_s(\psi(x_{i,n+1}),\psi(x_{in}))< 1/2^{n+4}$, so that
$$
x_{i,n+1}-x_{in} = p_{in}-q_{in}
$$
for some $p_{in},q_{in} \in G^+$ satisfying $|p_{in}+q_{in}|_s< 1/2^{n+4}$. Similarly, each
$$
y_{j,n+1}-y_{in} = r_{jn}-s_{jn}
$$
for some $r_{jn},s_{jn} \in G^+$ satisfying $|r_{jn}+s_{jn}|_s< 1/2^{n+4}$.
Set
$$
a_n:= p_{1n}+q_{1n}+p_{2n}+q_{2n}+  r_{1n}+s_{1n} + r_{2n}+s_{2n}
$$
for all $n$. Then $|a_n|_s < 4 /2^{n+4}=1/2^{n+2}$. Moreover,

\begin{align*}
x_{in}-a_n &\le x_{in}-q_{in} = x_{i,n+1}-p_{in} \le x_{i,n+1}\\
&\le x_{i,n+1} +q_{in} = x_{in}+p_{in} \le x_{in} +a_n,\\
y_{jn}-a_n &\le y_{jn}-s_{jn} = y_{j,n+1}-r_{jn} \le y_{j,n+1}\\
&\le y_{j,n+1} +s_{jn} = y_{jn}+r_{jn} \le y_{jn} +a_n
\end{align*}
for all $i,j,n$.

Next, we construct elements $b_1,b_2,\ldots $ in $G^+$ such that $|b_n|_s < 1/2^{n+1}$ for all $n$, while also $x_{in}\le y_{jn}+b_n$ for all $i,j,n$.

Fix $n$ for a while. Since each $y_j-x_i$ lies in $\overline G^+$, we have
$$
\overline d_s(\psi(t_{ij}),y_j-x_i)< 1/2^{n+4}
$$
for some $t_{ij} \in G^+$. Then

\begin{align*}
|t_{ij}- (y_{jn}-x_{in})|_s&= \overline d_s(\psi(t_{ij}),\psi(y_{jn})-\psi(x_{in}))\\
&\le \overline d_s(\psi(t_{ij}),y_{j}-x_{i}) +\overline d_s(y_j,\psi(y_{jn})) + \overline d_s(\psi(x_{in}),x_i)\\
&< 1/2^{n+4}+ 1/2^{n+5} + 1/2^{n+5}= 1/2^{n+3},
\end{align*}
and consequently,
$$
t_{ij}-y_{jn}+x_{in}=u_{ij}-v_{ij}
$$
for some $u_{ij},v_{ij}\in G^+$ satisfying $|u_{ij}+v_{ij}|_s< 1/2^{n+3}$. Set
$$
b_n:= u_{11}+u_{12}+u_{21}+u_{22},
$$
then $|b_n|_s\le \sum |u_{ij}+v_{ij}|_s < 4/2^{n+3}=1/2^{n+1}$. Moreover,
$$
x_{in}\le x_{in}+t_{ij}=y_{jn}+u_{ij}-v_{ij}\le y_{jn}+u_{ij} \le y_{jn}+b_n
$$
for all $i,j$.

Finally, we construct elements $z_1,z_2,\ldots $ in $G$ such that
$$
x_{in}\le z_n\le y_{ij}+b_n
$$
for all $i,j,n$, while also $|z_{n+1}-z_n|_s \le 1/2^n$.

As $x_{i1}\le y_{j1}+ b_1$ for all $i,j$, interpolation in $G$ immediately provides us an element $z_1$. Now suppose that $z_1,\ldots,z_n$ have been constructed, for some $n$. Then
\begin{align*}
x_{i,n+1}\le y_{j,n+1}+b_{n+1},\quad x_{i,n+1}\le x_{in}+a_n \le z_n +a_n,\\
z_n-b_n-a_n \le y_{jn} - a_n \le y_{j,n+1} \le y_{j,n+1} +b_{n+1}
\end{align*}
for all $i,j$. Hence, there exists $z_{n+1}\in G$ such that
\begin{eqnarray*}
x_{i,n+1} &  \le z_{n+1} &\le  y_{1,n+1}+ b_{n+1}\\
x_{2,n+1} &\le z_{n+1} &\le y_{2,n+1} + b_{n+1}\\
z_n-b_n -a_n & \le z_{n+1}  & \le  z_n +a_n.
\end{eqnarray*}
Since $-(a_n+b_n)\le z_{n+1}-z_n \le a_n \le a_n+b_n$, we conclude from property (v) of $|\cdot|_s$ that
$$
|z_{n+1}-z_n|_s \le |a_n +b_n|_s = |a_n|_s + |b_n|_s < 1/2^{n+2}+1/2^{n+1}< 1/2^n
$$
which completes the induction.

The sequence $\{z_n\}$ is a Cauchy sequence in $G$, and hence, there is $z \in \overline G$ such that $\psi(z_n)\to z$. In view of
$$
\overline d_s(\psi(b_n),0)= |b_n|_s < 1/2^{n+1}
$$
for all $n$, we also have $\psi(b_n) \to 0$. Since
$$
\psi(x_{in}) \le \psi(z_n) \le \psi(y_{jn})+\psi(b_n)
$$
for all $i,j,n$, we have finally $x_i \le z \le y_j$ for all $i,j$ which proves that $\overline G$ has interpolation.
\end{proof}

\begin{theorem}\label{th:4.6}
Let the conditions of Proposition {\rm \ref{pr:4.4}} hold. Then $\overline G$ is a Dedekind complete $\ell$-group.
\end{theorem}

\begin{proof}
Let $x,y \in \overline G$. Let $A$ be the set of lower bounds for $\{x,y\}$. Then $A$ is a non-empty set. In view of Proposition \ref{pr:4.5}, $\overline G$ has interpolation, so that $A$ is an upwards directed set, and therefore, if $A$ is indexed by itself, $A$ satisfies condition of Proposition \ref{pr:4.4}, so that $A$ has supremum $a$ in $\overline G$, and clearly, $a = x\wedge y$.  Similarly, $(-x)\wedge (-y)$ exists in $\overline G$, and $-((-x)\wedge (-y))=x\vee y$ exists in $\overline G$ proving $\overline G$ is an $\ell$-group. Applying Proposition \ref{pr:4.4}, we see $\overline G$ is a Dedekind complete $\ell$-group.
\end{proof}

As an important corollary of the latter theorem we have that if $(R,1_R)= (\mathbb R^n,1_{\mathbb R^n})$, $n \ge 1$, then the metrical completion $\overline G$ of $G$ with respect to any $(R,1_R)$-state is a Dedekind complete $\ell$-group which generalizes \cite[Thm 12.7]{Goo}:

\begin{theorem}\label{th:4.7}
Let $s$ be any $(\mathbb R^n,1_{\mathbb R^n})$-state on a unital $\ell$-group $(G,u)$, $n\ge 1$. There is a metrical completion $\overline G$ of $G$ with respect to $s$ such that $\overline G$ is a Dedekind complete $\ell$-group.
\end{theorem}

\begin{proof}
By Remark \ref{re:4.2}, the metrical completion $\overline G$ of $G_s$ with respect to $\widehat {\tilde s}$ is in fact a metrical completion of $G$ with respect to $s$.  The desired result follows from Theorem \ref{th:4.6}.
\end{proof}

\begin{theorem}\label{th:4.8}
Let the conditions of Proposition {\rm \ref{pr:4.4}} hold with an $(R,1_R)$-state $s$ on $(G,u)$ and let $(R,1_R)=(C(T),1_T)$, where $T\ne \emptyset$ is a Hausdorff compact extremally disconnected topological space. If $G_0$ is an $\ell$-subgroup of $\overline G$ generated by $\psi(u)$, then the restriction $\overline s_0$ of $\overline s$ onto $G_0$ is an $(R,1_R)$-state on the unital Dedekind complete $\ell$-group $(G_0,\psi(u))$, where $\overline s$ is a continuous mapping defined in Proposition {\rm \ref{pr:4.1}(ii)}. In addition, $\overline s_0$ is an extremal $(R,1_R)$-state on $(G_0,\psi(u))$ if and only if so is $s$ on $(G,u)$.
\end{theorem}

\begin{proof}
Let $s$ be an $(R,1_R)$-state on $(G,u)$ and let $G_0$ be the $\ell$-subgroup of $\overline G$ generated by $\psi(u)$. Due to Theorem \ref{th:4.6}, $(G_0,\psi(u))$ is an Abelian Dedekind complete unital $\ell$-group.  By Proposition \ref{pr:4.1}(ii), there is a unique continuous mapping $\overline s: \overline G\to (R,1_R)=(C(T),1_T)$ such that $\phi\circ s = \overline s\circ \psi $ (the mapping $\phi: R \to C(T)$ is the identity). Consequently, $\overline s_0$ is an $(R,1_R)$-state on $(G_0,\psi(u))$.

Assume that $s$ is an extremal $(R,1_R)$-state and let $\overline s_0=\lambda m_1 +(1-\lambda) m_2$, where $m_1,m_2$ are $(R,1_R)$-states on $(G_0,\psi(u))$ and $\lambda \in (0,1)$. The mappings $s_i(x):=m_i(\psi(x))$, $x \in G$, are $(R,1_R)$-states on $(G,u)$ for each $i=1,2$, and $s(x)=\lambda s_1(x)+(1-\lambda)s_2(x)$, $x \in G$. The extremality of $s$ entails $s(x)=s_1(x)=s_2(x)$ for each $x \in G$. We have to show that $\overline s_0=m_1=m_2$.

Since $\overline s_0(g), m_1(g),m_2(g)$ are in fact continuous functions on $T$, then they are positive functions for each $g \in \overline G^+_0$, and hence, $\overline s_0(g)/\lambda \ge m_1(g)$ and $\overline s_0(g)/ (1-\lambda)\ge m_2(g)$ for each $g \in \overline G_0^+$.
For any $g \in G^+_0$, there is a sequence $\{x_n\}$ in $G^+$ such that $\psi(x_n)\le \psi(x_{n+1})\le g$ and $\psi(x_n)\to g$. Then
$\|m_1(g)-m_1(\psi(x_n))\|_T\le \|\overline s(g)- \overline s(\psi(x_n))\|_T/\lambda \to 0$ and whence, $m_1(g)=\lim_n m_1(\psi(x_n))= \lim_n s_1(x_n)= \lim_n s(x_n)=\overline s(g)$. In a similar way we have $m_2(g)=\lim_n s_2(x_n)= \lim_n s(x_n)=\overline s(g)$. Then $m_1(g)=m_2(g)=\overline s(g)=\overline s_0(g)$ for each $g \in  G_0$ which shows that $\overline s_0$ is an extremal $(R,1_R)$-state on $(G_0,\psi(u))$.

Conversely, let $\overline s_0$ be an extremal $(R,1_R)$-state on $(G_0,\psi(u))$ and let $s=\lambda s_1 +(1-\lambda)s_2$, where $s_1,s_2$ are $(R,1_R)$-states on $(G,u)$ and $\lambda \in (0,1)$. We define mappings $m_i: G_0\to R$ for $i=1,2$ as follows. First, we put $m_i(\psi(x))=s_i(x)$ for $x \in G^+$ and $i=1,2$. Then each $m_i$ is a well-defined mapping on $\psi(G)^+$. Now let $g \in G^+_0$. There is a sequence $\{x_n\}$ in $G^+$ with $\psi(x_n) \le \psi(x_{n+1})\le g$ such that $g= \lim_n \psi(x_n)$. Since for continuous functions we have $0\le \psi(x_m)-\psi(x_n)$ for each $m\ge n$, then $\|s_1(\psi(x_m))-s_1(\psi(x_n))\|_T \le \|\overline s_0(\psi(x_m))-\overline s_0(\psi(x_n))\|_T/ \lambda \to 0$ and $\|s_2(\psi(x_m))-s_2(\psi(x_n))\|_T \le \|\overline s_0(\psi(x_m))-\overline s_0(\psi(x_n))\|_T/ (1-\lambda) \to 0$ when $m,n\to \infty$. Then $\{s_i(\psi(x_n))\}$ is a Cauchy sequence in $C(T)= R$, and there is $f_i\in C(T)^+$ such that $s_i(\psi(x_n)) \rightrightarrows f_i$ for $i=1,2$. If $\{y_n\}$ in $G^+$ is another sequence in $G^+_0$ such that $\psi(y_n) \le \psi(y_{n+1})\le g$ and $g= \lim_n \psi(y_n)$, then $\lim_ns_i((\psi(x_n)))=f_i=\lim_n s_i(\psi(y_n))$, and therefore, the extension of $m_i$ to $G^+_0$ is defined by $m_i(g):=\lim_n m_i(\psi(x_n))$ whenever $\{x_n\}$ is a sequence in $G^+$ with $\psi(x_n) \le \psi(x_{n+1})\le g$ such that $g= \lim_n \psi(x_n)$. Finally, $m_i$ can be extended to the whole $G_0$, so that every $m_i$ is an $(R,1_R)$-state on $(G,u)$, and $\overline s_0= \lambda m_1 +(1-\lambda)m_2$. This yields $\overline s_0=m_1=m_2$ and consequently, $s=s_1=s_2$ proving $s$ is an extremal $(R,1_R)$-state on $(G,u)$.
\end{proof}

\section{Lattice Properties of $R$-measures and Simplices}

In this section we extend the notion of an $(R,1_R)$-state to $R$-measures and $R$-Jordan signed measures on a pseudo MV-algebra. If $R$ is a Dedekind complete Riesz space, we show that the space of $R$-Jordan signed measures can be converted into a Dedekind complete Riesz space. This allows us to show when the space of $(R,1_R)$-states on a pseudo MV-algebra is a Choquet simplex or even a Bauer simplex. In addition, we show when every state is a weak limit of a net of convex combinations of $(R,1_R)$-state-morphisms.

Thus let $M$ be a pseudo MV-algebra and $R$ be a Riesz space.  A mapping $m:M \to R$ is said to be an $R${\it -signed measure} if $m(x+y)=m(x)+m(y)$ whenever $x+ y$ is defined in $M$. An $R$-signed state is (i) an $R${\it -measure} if $m(x)\ge 0$ for each $x \in M$, (ii) an $R${\it -Jordan signed measure} if $m$ is a difference of two $R$-measures. It is clear that (i) every $(R,1_R)$-state is an  $R$-measure, (ii) $m(0)=0$ for each $R$-signed measure $m$, (iii) if $x\le y$, then $m(x)\le m(y)$ whenever $x\le y$ for each $R$-measure $m$. We denote by $\mathcal{JSM}(M,R)$ and $\mathcal M(M,R)$ the set of $R$-Jordan signed measures and $R$-measures, respectively, on $M$. Then $\mathcal{JSM}(M,R)$ is a real vector space and if for two $R$-Jordan signed measures $m_1$ and $m_2$ we put $m_1\le^+m_2$, then $\mathcal{JSM}(M)$ is an Abelian po-group with respect to the partial order $\le^+$ with positive cone $\mathcal M(M,R)$. Using ideas from \cite[p. 38--41]{Goo}, we show that $\mathcal{JSM}(M,R)$ is a Dedekind complete Riesz space whenever $R$ is a Dedekind complete Riesz space. We note that in \cite{Goo} this was established for Abelian interpolation po-groups $G$ whereas we have functions on $M$ with the partial operation $+$ that is not assumed to be commutative a priori.

In this section, let $R$ be a Dedekind complete Riesz space and $M$ be a pseudo MV-algebra.

A mapping $d:M\to  R$ is said to be {\it subadditive}
provided $d(0) = 0$ and $d(x+y)\le d(x)+d(y)$ whenever  $x+y \in M$.

\begin{proposition}\label{pr:3.1}
Let $M$ be a pseudo MV-algebra, $R$ a Dedekind complete Riesz space, and let $d:M\to R$ be a subadditive mapping. For all $x\in M$, assume that
the set
\begin{equation}\label{eq:D(x)}
D(x):=\{d(x_1)+\cdots+d(x_n): x = x_1+\cdots+x_n, \ x_1,\ldots,x_n
\in M,\ n\ge 1 \}
\end{equation}
is bounded above in $R$.  Then there is an $R$-signed measure
$m:M\to R$ such that $m(x)=\bigvee D(x)$ for all $x\in M$.
\end{proposition}

\begin{proof}
The map $m(x):=\bigvee D(x)$ is a well-defined mapping for all $x
\in M$. It is clear that $m(0)=0$ and now we show that
$m$ is additive on $M$.

Let $x+y \in M$ be given. For all decompositions
$$ x = x_1 +\cdots+x_n \ \mbox{and} \ y=y_1+\cdots +y_k$$
with all $x_i,y_j \in M$, we have $x+y = x_1+\cdots+x_n + y_1+\cdots
+ y_k$, which yields

$$ \sum_i d(x_i)+\sum_j d(y_j) \le m(x+y).
$$
Therefore, $s+t \le m(x+y)$ for all $s,t\in D(x)$.
Since $R$ is a Dedekind complete Abelian $\ell$-group, $\bigvee$ is distributive with respect to $+$, see \cite[Prop 1.4]{Goo}. Whence
\begin{eqnarray*}
m(x)+m(y)&=& \left(\bigvee D(x)\right) +m(y) = \bigvee_{s \in D(x)} (s+m(y))\\
&=& \bigvee_{u\in D(x)} \left(s+ \left(\bigvee D(y)\right)\right) =
\bigvee_{s\in
D(x)}\bigvee_{y \in D(y)} (s+t)\\
&\le& m(x+y).
\end{eqnarray*}

Conversely, let $x+y=z_1+\cdots+z_n$ be a decomposition of $x+y$ where each $z_i \in M$. Then the strong Riesz decomposition Property
RDP$_2$ with (\ref{eq:RDP1})--(\ref{eq:RDP2}) implies that there are elements $x_1,\ldots,x_n, y_1,\ldots, y_n
\in M$ such that $x = x_1+\cdots+x_n$, $y = y_1+\cdots+y_n$ and $z_i
= x_i+y_i$ for $i=1,\dots,n$. This yields
$$
\sum_i d(z_i) \le \sum_i (d(x_i)+d(y_i)) = \left(\sum_i
d(x_i)\right) + \left(\sum_i d(y_i)\right) \le m(x)+m(y),
$$
and therefore, $m(x+y)\le m(x)+m(y)$ and finally, $m(x+y)=m(x)+m(y)$
for all $x,y \in M$ such that $x+y$ is defined in $M$, so that $m$
is an $R$-signed measure on $M$.
\end{proof}

\begin{theorem}\label{th:3.4} Let $M$ be a pseudo MV-algebra and $R$ be a Dedekind complete Riesz space. For the set $\mathcal J(M,R)$ of $R$-Jordan signed measures on $M$ we have:

\begin{enumerate}

\item[(a)] $\mathcal J(M,R)$ is a Dedekind complete $\ell$-group with respect to the partial order $\le^+$.

\item[(b)] If $\{m_i\}_{i\in I}$ is a non-empty set of $\mathcal J(M,R)$
that is bounded above, and if $d(x)=\bigvee_i m_i(x)$ for all $x \in
M$, then
$$ \left(\bigvee_i m_i\right)(x) = \bigvee\{d(x_1)+\cdots + d(x_n):
x= x_1+\cdots + x_n, \ x_1,\ldots, x_n \in M\}
$$
for all $x \in M$.

\item[(c)] If $\{m_i\}_{i\in I}$ is a non-empty set of $\mathcal J(M,R)$
that is bounded below, and if $e(x)=\bigwedge_i m_i(x)$ for all $x
\in M$, then
$$ \left(\bigwedge_i m_i\right)(x) = \bigwedge\{e(x_1)+\cdots + e(x_n):
x= x_1+\cdots + x_n, \ x_1,\ldots, x_n \in M\}
$$
for all $x \in M$.
\item[(d)] The set $\mathcal J(M,R)$ is a Dedekind complete Riesz space.

\end{enumerate}
\end{theorem}

\begin{proof} Let $m_0 \in \mathcal J(M,R)$ be an upper bound for $\{m_i\}_{i\in I}$. For any $x \in M$, we have $m_i(x)\le m_0(x)$, so that the mapping $d(x)=\bigvee_i m_i(x)$ defined on $M$ is  a subadditive mapping on the pseudo MV-algebra $M$.
For any $x \in M$ and any decomposition $x = x_1+\cdots + x_n$
with all $x_i \in M$, we conclude $d(x_1)+\cdots+ d(x_n)\le
m(x_1)+\cdots + m(x_n)\le m_0(x)$.  Hence, $m_0(x)$ is an upper bound for
$D(x)$ defined by (\ref{eq:D(x)}).

By Proposition \ref{pr:3.1}, we conclude that there is an $R$-signed measure $m:M \to R$ such that $m(x)=\bigvee D(x)$. For every $x \in M$ and
every $m_i$ we have $m_i(x)\le d(x)\le m(x)$, which gives $m_i \le^+
m$. The mappings $m-m_i$ are positive $R$-measures belonging to
$\mathcal J(M,R)$, so that $m-m_i=:f_i\in \mathcal M(M,R)$, and $m = m_i^++f_i - m_i^-$, where $m_i^+,m_i^-\in \mathcal M(M,R)$ and $m_i=m_i^+ -m_i^-$.
Consequently, $m \in \mathcal J(M,R)$. If $h \in \mathcal J(M,R)$ is an $R$-Jordan signed measure such that $m_i\le^+ h$ for any $i\in I$, then $d(x)\le h(x)$ for any $x \in
M$. As in the preceding paragraph, we can show that $h(x)$ is also an upper bound for $D(x)$, whence $m(x)\le h(x)$ for any $x \in M$, which gives $m\le
^+h$. In other words, we have proved that $m$ is the supremum of
$\{m_i\}_{i\in I}$, and its form is given by (b).

Now if we apply  the order anti-automorphism $z\mapsto - z$ holding in the Riesz space $R$, we see that if the set $\{m_i\}_{i\in I}$  in $\mathcal J(M,R)$ is bounded below, then it has an infimum given by (c).

It is clear that $\mathcal J(M,R)$ is directed.  Combining (b)
and (c), we see that $\mathcal J(M,R)$ is a Dedekind complete $\ell$-group.

(d) If $m$ is an $R$-Jordan signed measure on $M$ and $\alpha \in \mathbb R$, then clearly $\alpha m\in \mathcal J(M,R)$ and if, in addition $\alpha \ge 0$, then $\alpha m$ an $R$-measure whenever $m$ is an $R$-measure. Consequently, $\mathcal J(M,R)$ is a Dedekind complete Riesz space.
\end{proof}

Now let, for an Archimedean unital Riesz space $(R,1_R)$, $(C(T),1_T,\phi)$ be its canonical representation, i.e. $(R,1_R)\sim (C(T),1_T,\phi)$. We say that a net of $(R,1_R)$-states $\{s_\alpha\}_\alpha$ on a pseudo MV-algebra $M$ {\it converges weakly} to an $(R,1_R)$-state $s$ on $M$, and we write $\{s_\alpha\}_\alpha\stackrel{w} \rightarrow s$, if $\|\phi\circ s_\alpha(x)- \phi\circ s(x)\|_T\to 0$ for each $x \in M$.
We note that the weak convergence introduced in the proof of (4) of Proposition \ref{pr:3.19} is a special case of the present definition.

If $(R,1_R)=(\mathbb R,1)$, then $(R,1_R)$-states are usual states on pseudo MV-algebras, therefore, the weak convergence of $(R,1_R)$-states coincides with the weak convergence of states introduced in the beginning of Section 3.

First, we show that if, for a net of $(R,1_R)$-states on $M$, we have $\{s_\alpha\}_\alpha\stackrel{w} \rightarrow s$ and $\{s_\alpha\}_\alpha\stackrel{w} \rightarrow s'$, then $s=s'$.
Indeed, if $s'$ is another $(R,1_R)$-state on $M$ such that $\|\phi\circ s_\alpha(x)- \phi\circ s'(x)\|_T\to 0$ for each $x \in M$, then $\|\phi(s(x)-s'(x))\|_T\le \|\phi(s(x)-s_\alpha(x))\|_T+\|\phi(s_\alpha(x)-s'(x))\|_T\to 0$ so that $\phi(s(x))=\phi(s'(x))$ which proves $s(x)=s'(x)$ for each $x \in M$ and finally, we have $s=s'$.

We note that the weak convergence of $(R,1_R)$-states on $M$ can be defined also in another but equivalent form: Let $(R,1_R)$ be an Archimedean unital Riesz space. For any $r\in R$, we set
$$\|r\|_{1_R}:=\inf\{\alpha\in \mathbb R^+\colon |r|\le \alpha 1_R\}.
$$
Then $\|\cdot\|_{1_R}$ is a norm on $R$. In particular, for  each $f\in C(T)$, we have $\|f\|_T=\|f\|_{1_T}$.
In addition, if $(R,1_R)\sim (C(T),1_T,\phi)$, then $\|x\|_{1_R}=\|\phi(x)\|_T$ for each $x \in R$. Therefore,
a net $\{s_\alpha\}_\alpha$ of $(R,1_R)$-states converges weakly to an $(R,1_R)$-state $s$ iff $\lim_\alpha\|s_\alpha(x)-s(x)\|_{1_R}=0$ for each $x \in M$.

\begin{proposition}\label{pr:conv2}
Let $M$ be a pseudo MV-algebra and $(R,1_R)$ be a Dedekind $\sigma$-complete unital Riesz space. Then the space $\mathcal S(M,R,1_R)$ is either empty or a non-empty convex compact set under the weak convergence.
\end{proposition}

\begin{proof}
By Proposition \ref{pr:3.5}, $M$ has at least one state iff $M$ possesses at least one normal ideal that is also maximal. In particular, if $M$ is an MV-algebra, with $0\ne 1$, $M$ admits at least one state.

Thus, let $M$ have at least one $(R,1_R)$-state. Clearly, $\mathcal S(M,R,1_R)$ is a convex set. Since $R$ is Dedekind $\sigma$-complete, according to \cite[Thm 45.4]{LuZa}, see also Theorem \ref{th:sigma}, $(R,1_R)$ has the canonical representation $(C(T),1_T,\phi)$, and $\phi$ is bijective, where $T$ is the set of maximal ideals of $(R,1_R)$ with the hull-kernel topology.

Let $D:=\{f\in C(T)\colon \|f\|_T\le 1\}$. If $s$ is an $(R,1_R)$-state, then $\phi\circ s\in D^M$. Since $D$ is compact in the norm-topology $\|\cdot\|_T$, $D^M$ is due to Tychonoff's theorem a compact Hausdorff topological space in the product topology of $D^M$.
The set $\phi(\mathcal S(M,R,1_R)):=\{\phi\circ s\colon s \in \mathcal S(M,R,1_R)\}$ is a subset of the cube $D^M$. Let us assume that $\{s_\alpha\}_\alpha$ is a net of $(R,1_R)$-states on $M$
such that there exists the limit $\mu(x)=\lim_\alpha \phi\circ s_\alpha(x)\in D\subset C(T)$ for each $x \in M$. Then $\mu: x \mapsto \mu(x)$, $x \in M$, is a $(C(T),1_T)$-state on $M$. Put $s(x):=\phi^{-1}(\mu(x))$ for each $x \in M$. Then $s:x\mapsto s(x)$, $x \in M$, is an $(R,1_R)$-state on $M$ such that $\{s_\alpha\}_\alpha\stackrel{w} \rightarrow s$, which says that $\phi(\mathcal S(M,R,1_R))$ is a closed subset of $D^M$. Since $D$ is compact in the norm-topology $\|\cdot\|_T$, and $\phi(\mathcal S(M,R,1_R))$ is a closed subset of $D^M$, $\phi(\mathcal S(M,R,1_R))$ is compact. Consequently, $\mathcal S(M,R,1_R)$ is a compact set in the weak topology of $(R,1_R)$-states.
\end{proof}

\begin{corollary}\label{co:KM1}
Under the conditions of Proposition {\rm \ref{pr:conv2}} every $(R,1_R)$-state on $M$ lies in the closure of the convex hull of extremal $(R,1_R)$-states on $M$, where the closure is given in the weak topology of $(R,1_R)$-states, i.e.
$$ \mathcal S(M,R,1_R)= (\Con(\mathcal S_\partial(M,R,1_R)))^-.$$
\end{corollary}

\begin{proof}
It is a direct application of the Krein--Mil'man Theorem, Theorem \cite[Thm 5.17]{Goo}, and Proposition \ref{pr:conv2}.
\end{proof}

\begin{proposition}\label{pr:conv3}
Let $M$ be a pseudo MV-algebra and $(R,1_R)$ be a unital Riesz space isomorphic to the unital Riesz space $(C_b(T),1_T)$ of bounded real-valued functions on $T$, where $T\ne \emptyset$ is a basically disconnected compact Hausdorff topological space. Then the set of extremal $(R,1_R)$-states on $M$ is closed in the weak topology of $(R,1_R)$-states.
\end{proposition}

\begin{proof}
According to Proposition \ref{pr:3.15}, every extremal $(R,1_R)$-state on $M$ is an $(R,1_R)$-state-morphism on $M$ and vice-versa. Since $T$ is basically disconnected, by Nakano's theorem $(C_b(T),1_T)$ is a Dedekind $\sigma$-complete Riesz space, consequently, so is $(R,1_R)$. By Proposition \ref{pr:conv2}, we can introduce the weak topology of $(R,1_R)$-states on $M$ which gives a compact space $\mathcal S(M,R,1_R)$. Applying the criterion (iii) of Proposition \ref{pr:3.15}, we see that the space $\mathcal{SM}(M,R,1_R)$ of $(R,1_R)$-state-morphisms is closed and compact. Due to (\ref{eq:part}), we have $\mathcal S_\partial(M,R,1_R)=\mathcal{SM}(M,R,1_R)$ is also compact.
\end{proof}

\begin{corollary}\label{co:KM2}
Under the conditions of Proposition {\rm \ref{pr:conv3}} every $(R,1_R)$-state on $M$ lies in the closure of the convex hull of $(R,1_R)$-state-morphisms on $M$, where the closure is given in the weak topology of $(R,1_R)$-states, i.e.
$$ \mathcal S(M,R,1_R)= (\Con(\mathcal{SM}(M,R,1_R)))^-.$$
\end{corollary}

\begin{proof}
Due to (\ref{eq:part}), we have $\mathcal S_\partial(M,R,1_R)=\mathcal{SM}(M,R,1_R)$. Applying Corollary \ref{co:KM1}, we have the result.
\end{proof}

Now we present some results when the space of $(R,1_R)$-states on a pseudo MV-algebra is a Choquet simplex or even a Bauer simplex. Therefore, we introduce some notions about simplices. For more info about them see the books \cite{Alf,Goo}.

We recall that a {\it convex cone} in a real linear space $V$ is any
subset $C$ of  $V$ such that (i) $0\in C$, (ii) if $x_1,x_2 \in C$,
then $\alpha_1x_1 +\alpha_2 x_2 \in C$ for any $\alpha_1,\alpha_2
\in \mathbb R^+$.  A {\it strict cone} is any convex cone $C$ such
that $C\cap -C =\{0\}$, where $-C=\{-x:\ x \in C\}$. A {\it base}
for a convex cone $C$ is any convex subset $K$ of $C$ such that
every non-zero element $y \in C$ may be uniquely expressed in the
form $y = \alpha x$ for some $\alpha \in \mathbb R^+$ and some $x
\in K$.

Any strict cone $C$ of $V$ defines a partial order $\le_C$ on $V$ via $x
\le_C y$ iff $y-x \in C$. It is clear that $C=\{x \in V:\ 0 \le_C
x\}$. A {\it lattice cone} is any strict convex cone $C$ in $V$ such
that $C$ is a lattice under $\le_C$.

A {\it simplex} in a linear space $V$ is any convex subset $K$ of
$V$ that is affinely isomorphic to a base for a lattice cone in some
real linear space. A  simplex $K$ in a locally convex Hausdorff
space is said to be (i) {\it Choquet} if $K$ is compact, and (ii)
{\it Bauer} if $K$ and $K_\partial $ are compact, where $K_\partial
$ is the set of extreme points of $K$.

\begin{theorem}\label{th:Choq1}
Let $M$ be a pseudo MV-algebra and $(R,1_R)$ be a Dedekind complete unital Riesz space. Then the set of $(R,1_R)$-states on $M$ is either empty set or a non-void Choquet simplex.
\end{theorem}

\begin{proof}
By Proposition \ref{pr:3.5}, $M$ has at least one $(R,1_R)$ state iff $M$ has at least one normal ideal that is also normal. Thus assume that $M$ admits at least one $(R,1_R)$-state. According to Theorem \ref{th:3.4}, the space $\mathcal J(M,R)$ of $R$-Jordan signed measures on $M$ is a Dedekind complete Riesz space. Since the positive cone of $\mathcal J(M,R)$ is the set $\mathcal M(M,R)$ of $R$-measures on $M$ that is also a strict lattice cone of $\mathcal J(M,R)$, it is clear that the set $\mathcal S(M,R,1_R)$ of $(R,1_R)$-states is a base for $\mathcal J(M,R)$. Whence, $\mathcal S(M,R,1_R)$ is a simplex. Now applying Proposition \ref{pr:conv2}, we see that $\mathcal S(M,R,1_R)$ is compact in the weak topology of $(R,1_R)$-states, which gives the result.
\end{proof}

Something more we can say when $(R,1_R)\cong (C_b(T),1_T)$ for some extremally disconnected space $T$.

\begin{theorem}\label{th:Choq2}
Let $M$ be a pseudo MV-algebra and $(R,1_R) \cong(C_b(T),1_T)$ for some extremally disconnected space $T\ne \emptyset$.  Then $\mathcal S(M,R,1_R)$ is either the empty set or a non-void a Bauer simplex.
\end{theorem}

\begin{proof}
Assume that $M$ possesses at least one $(R,1_R)$-state. Due to the Nakano theorem, $(C_b(T),1_T)$ is a unital Dedekind complete Riesz space, consequently, so is $(R,1_R)$. Applying Theorem \ref{th:Choq1}, we have $\mathcal S(M,R,1_R)$ is compact and by Proposition \ref{pr:conv3}, the space of extremal states is compact in the weak topology of $(R,1_R)$-states, so that $\mathcal S(M,R,1_R)$ is a Bauer simplex.
\end{proof}

\section{Conclusion}

In the paper, we have introduced $(R,1_R)$-states on pseudo MV-algebras, where $R$ is a Riesz space with a fixed strong unit $1_R$, as additive functionals on a pseudo MV-algebra $M$ with values in the interval $[0,1_R]$ preserving partial addition $+$ and mapping the top element of $M$ onto $1_R$. $(R,1_R)$-states generalize usual states because every $(\mathbb R,1)$-state is a state and vice versa. Besides we have introduced $(R,1_R)$-state-morphisms and extremal $(R,1_R)$-states. If $(R,1_R)$ is an Archimedean unital Riesz space, every $(R,1_R)$-state-morphism is an extremal $(R,1_R)$-state, Theorem \ref{th:3.17}. We note that there are $(R,1_R)$-state-morphisms whose kernel is not maximal ideal, Proposition \ref{pr:3.19}, whereas, if an $(R,1_R)$-state has a maximal ideal, it is an $(R,1_R)$-state-morphism, Proposition \ref{pr:3.7}.
Metrical completion of a unital $\ell$-group with respect to an $(R,1_R)$-state, when $(R,1_R)$ is a Dedekind complete unital Riesz space, gives a Dedekind complete $\ell$-group, Theorem \ref{th:4.6}. Theorem \ref{th:3.4} shows that the space of $R$-Jordan signed measures, when $R$ is a Dedekind complete Riesz space, can be converted into a Dedekind complete Riesz space. This allows us to show when the space $(R,1_R)$-states is a compact set, Proposition \ref{pr:conv2}, and when every $(R,1_R)$-state is in the weak closure of the convex hull of extremal $(R,1_R)$-states, Corollary \ref{co:KM2}. We have showed that the space of $(R,1_R)$-states, when $(R,1_R)$ is Dedekind complete, is a Choquet simplex, Theorem \ref{th:Choq1}, and we proved when it is even a Bauer simplex, Theorem \ref{th:Choq2}.

\end{document}